\documentclass[12 pt, a4paper twoside]{amsart}

\usepackage{amsthm,amssymb,amsmath,amsfonts,mathrsfs,amscd}
\usepackage[latin1]{inputenc}
\usepackage{mathtools}
\usepackage[all,2cell]{xy}
\usepackage[margin=1in, a4paper]{geometry}
\usepackage{stmaryrd}
\usepackage{comment}
\usepackage{tikz-cd}
\usepackage{hyperref}

\theoremstyle{plain}
\newtheorem{theorem}{Theorem}

\newtheorem{lemma}{Lemma}

\newtheorem{proposition}{Proposition}
\newtheorem*{claim}{Claim}

\theoremstyle{definition}

\newcommand{\fr}{\mathfrak}

\def \Q{\mathbb{Q}}

\def \C {\mathbb{C}}

\def \R {\mathbb{R}}

\def\H {\mathbb H}
\def\Z{\mathbb Z}

\def \qpbar {\overline{\Q}_p}

\def\gp{\mathfrak{p}}

\def \calh {\mathcal{H}}

\def \call {\mathcal{L}}

\def \calo {\mathcal{O}}

\def\id{1\!\!1}

\def\Hom{\text{Hom}}

\def\GL{\text{GL}}
\def\SL{\mathrm{SL}}

\def\SU{\mathrm{SU}}

\def\det{\text{det}}
\def\Tr{\text{Tr}}
\def\Nm{\text{Nm}}

\def\f{\mathbf{f}}
\def\ulk{\underline k}
\def\uln{\underline n}
\def\ulz{\underline z}

\title{$p$-adic Asai $L$-functions attached to Bianchi cusp forms}

\author[B. Balasubramanyam]{Baskar Balasubramanyam}
\address{
Indian Institute of Science Education and Research \\
Pashan, Pune 411021, India.
}
\email{baskar@iiserpune.ac.in}

\author[E. Ghate]{Eknath Ghate}
\address{
School of Mathematics \\
Tata Institute of Fundamental Research \\
Homi Bhabha Road, Mumbai 400005, India.
}
\email{eghate@math.tifr.res.in}
\author[Ravitheja Vangala]{Ravitheja Vangala} 
\address{
School of Mathematics \\
Tata Institute of Fundamental Research \\
Homi Bhabha Road, Mumbai 400005, India.
}
\email{ravithej@math.tifr.res.in}
\begin{document}
\begin{abstract}
We establish a rationality result for the twisted Asai $L$-values attached to a Bianchi cusp form and construct distributions interpolating 
these $L$-values. Using the method of abstract Kummer congruences, we then outline the main steps needed to show that these distributions 
come from a measure.
\end{abstract}
\maketitle

\section{Preliminaries}

Let $F$ be an imaginary quadratic field with ring of integers $\calo_F$. Write $F= \Q(\sqrt{-D})$ with $D>0$ and $-D$ the discriminant of $F$. 
Let $S_{\ulk} (\fr n)$ denote the space of Bianchi cusp forms of weight $\ulk = (k,k)$, $k \geq 2$, and level $\fr n$ and central character with 
trivial finite part and infinity type $(2-k,2-k)$.   
Let $\f \in S_{\ulk} (\fr n)$ be a normalized eigenform and let $c(\fr m, \f)$ be the Fourier coefficients of $\f$, for any integral ideal $\fr m \subset \calo_F$. The eigenform $\f$ corresponds to a tuple $(f_1, \dots, f_h)$ of classical Bianchi cusp forms, where $h$ is the class number of $F$. We take $f = f_1 \in S_{\ulk} (\Gamma_0 (\fr n))$ and only focus on this since the Asai $L$-function depends only on $f_1$.

Let $c(r)$, for $r \geq 1$, denote $c((r), f)$. Define the Asai $L$-function of $f$ by the formula
$$
G(s, f) = L_N (2s-2k+2, \id) \sum_{r=1}^\infty \frac{c(r)}{r^s},
$$
where $N$ is the positive generator of the ideal $\fr n \, \cap \, \Z$ and $L_N (s, \id)$ is the  $L$-function attached to the trivial character modulo $N$. 
The special values of this function are investigated in \cite{Gh}. A generalization to cusp forms defined over CM fields can be found 
in \cite{Ghb}.

Let $p \in \Z$ be an odd prime integer that is relatively prime to $N$ and that is also unramified in $F$. 
Let $\chi : (\Z/p^j \Z)^\times \to \C^\times$ be a 
Dirichlet character with conductor dividing $p^j$. Define the twisted Asai $L$-function of $f$ by the formula
$$
G(s,\chi, f) = L_N (2s-2k+2, \chi^2)  \sum_{r=1}^\infty \frac{c(r) \chi(r)}{r^s}.
$$
This has an Euler product expansion
$$
G(s, \chi, f) = \prod_p {G_p (s, \chi, f)},
$$
where the local $L$-functions at all but finitely many primes are described as follows. Let $l \ne p$ be an integer prime, not dividing $N$. 
For any $\fr l | l$, let $\alpha_1 (\fr l)$ and $\alpha_2 (\fr l)$ denote the reciprocal roots of the Hecke polynomial of $f$ at 
$\fr l$: $1 - c (\fr l, f) X + \Nm (\fr l)^{k-1} X^2 $. Then 
$$
\frac{1}{G_l (s,\chi, f)} = \left\{ 
\begin{array}{ll}
\prod_{i,j} (1 - \chi (l) \alpha_i (\fr l) \alpha_j (\bar{\fr l}) l^{-s}) & \textrm{if}\ l = \fr l \bar{\fr l}, \\
(1 - \chi (l) \alpha_1 (\fr l) l^{-s}) (1 - \chi^2 (l) l^{-2s+2k-2}) (1 - \chi (l) \alpha_2 (\fr l) l^{-s})   & \textrm{if}\ l = \fr l \textrm{ is inert}, \\
(1 - \chi (l) \alpha^2_1 (\fr l) l^{-s}) (1 - \chi (l) l^{-s+k-1}) (1 - \chi (l) \alpha^2_2 (\fr l) l^{-s})   & \textrm{if}\ l = {\fr l}^2 \textrm{ is ramified}. \\
\end{array} \right.
$$

We want to find `periods' and prove that the special values of the twisted Asai $L$-functions are algebraic after dividing by these periods. We also want to $p$-adically interpolate the special values of $G(s,\chi,f)$ as $\chi$ varies over characters of $p$-power conductor.

\bigskip

\section{Complex valued distributions}

Following Panchishkin, we now construct a complex valued distribution that is related to the twisted Asai $L$-function. 
This section basically follows Coates--Perrin-Riou~\cite{Coates-PR} and Courtieu--Panchishkin~\cite[\S 1.6]{CouPan}.

The function $G(s,f)$ has an Euler product formula
$$
G(s, f) = \prod_p {G_p (s, f)} = \sum_{r=1}^\infty \frac{d(r)}{r^s},
$$
and hence satisfies the hypothesis in the above references. We now assume that our fixed prime $p$ splits as $\gp \bar \gp$ in $F$. 
A similar argument will also work for $p$ inert. Then the local Euler factor at $p$ is of the form $G_p (s, f) = F(p^{-s})^{-1}$ where
$$
F(X) =  (1 - \alpha_1 (\fr p) \alpha_1 (\bar{\fr p}) X) (1 - \alpha_1 (\fr p) \alpha_2 (\bar{\fr p}) X)
 (1 - \alpha_2 (\fr p) \alpha_1 (\bar{\fr p}) X) (1 - \alpha_2 (\fr p) \alpha_2 (\bar{\fr p}) X).
$$
In what follows we shall assume that
$f$ is totally ordinary at $p$. Hence we may assume, by possibly switching the subscripts $i = 1$, $2$, that the inverse 
root $\kappa := \alpha_1 (\fr p) \alpha_1 (\bar{\fr p})$ of the polynomial $F(X)$ is 
a $p$-adic unit. Also define a polynomial $H(X)$ as
$$
H(X) =  (1 - \alpha_1 (\fr p) \alpha_2 (\bar{\fr p}) X)
 (1 - \alpha_2 (\fr p) \alpha_1 (\bar{\fr p}) X) (1 - \alpha_2 (\fr p) \alpha_2 (\bar{\fr p}) X).
$$
Let $B_0 = 1$ and define $B_1, B_2$ and $B_3$ such that
$$
H(X) = 1 + B_1 X + B_2 X^2 + B_3 X^3.
$$ 

Let $\chi : (\Z/p^j\Z)^\times \to \C^\times$ be a character with conductor $C_\chi = p^{j_\chi}$.   We want to define a complex valued 
distribution that interpolates the values of the twisted $L$-function:
$$
G(s,\chi, f) =   \sum_{r=1}^\infty \frac{d(r) \chi (r)}{r^s}.
$$

Define functions $P_s: \Q \to \C$ by the formula
$$
P_s (b) = \sum_{r=1}^\infty \frac{d(r) e^{2 \pi i r b}}{r^{s}}
$$
which converges absolutely for $\Re(s)$ sufficiently large. Define a distribution $\tilde \mu$ on $\Z_p^\times$ by the formula
$$
\tilde \mu_s (a + p^j \Z_p) =  \frac{p^{j (s-1)}}{\kappa^j} \sum_{i=0}^3 B_i P_s (a p^i/p^j) p^{-is}.
$$ 
We need to check that this satisfies the distribution relations. We will do this by showing that
\begin{equation}\label{eq: eta-value}
\sum_{a \mod p^j} \chi (a) \tilde \mu_s (a + p^j \Z_p)
\end{equation}
is independent of $j$ as long as $j \ge j_\chi$. For any character $\chi$ and integer $M$, define the generalized Gauss sum
$$
G_{M, p^j} = \sum_{a \mod p^j} \chi (a) e^{2\pi i a M/p^j }.
$$
It can be verified that 
$$
G_{M, p^j} = \left\{\begin{array}{cc}
p^{j - j_\chi} G(\chi) \bar \chi (M/p^{j - j_\chi})  & \mathrm{if}\ p^{j - j_\chi} | M \\
0 & \mathrm{otherwise,} 
\end{array} \right.
$$
where $G(\chi) = G_{1, p^{j_\chi}}$ is the Gauss sum of $\chi$.

From the definition, we can write the quantity in equation \eqref{eq: eta-value} as
\begin{align}\label{eq: calc-step1}
&   \frac{p^{j (s-1)}}{\kappa^j} \sum_{a \mod p^j} \chi(a) \sum_{i=0}^3 B_i p^{-is} \sum_{r=1}^\infty d (r) e^{2 \pi i a p^i r/p^j} r^{-s} \nonumber \\
& =  \frac{p^{j (s-1)}}{\kappa^j}  \sum_{i, r} B_i p^{-is}  d (r) r^{-s} G_{p^i r, p^j} \nonumber \\
& =  \frac{p^{js -j_\chi}}{\kappa^j} G(\chi) \sum_{i, r} B_i p^{-is}  d (r) r^{-s} \bar \chi \left(\frac{p^i  r}{p^{j-j_\chi}}\right) \id_\Z \left(\frac{p^i r}{p^{j - j_\chi}}\right). 
\end{align}
Here $\id_\Z$ is the characteristic function of integers. It appears since the only terms that contribute to the sum are those with $p^{j - j_\chi} | p^i r$ (this follows from the above formula for $G_{p^i r, p^j}$).  Now write each $r = r_1 r_2$, where $r_1$ is the $p$-power part and $r_2$ is the away from $p$-part of $r$ respectively. We know that 
\begin{align*}
& \bar \chi \left(\frac{p^i r}{p^{j - j_\chi}}\right) = \bar \chi (r_2) \bar \chi \left(\frac{p^i r_1}{p^{j - j_\chi}}\right), \>\> \mathrm{ and} \\
& \id_\Z \left(\frac{p^i r}{p^{j-j_\chi}}\right) = \id_\Z \left(\frac{p^i r_1}{p^{j-j_\chi}}\right).
 \end{align*}
Using  these in equation \eqref{eq: calc-step1}, we get the expression
\begin{equation}\label{eq: calc-step2}
 \frac{p^{js - j_\chi}}{\kappa^j} G(\chi) \left(\sum_{r_2} \bar \chi (r_2) d (r_2) r_2^{-s}\right) \sum_{i, r_1} B_i p^{-is}  d (r_1) r_1^{-s} \bar \chi \left(\frac{p^i r_1}{p^{j - j_\chi}}\right) \id_\Z \left(\frac{p^i r_1}{p^{j - j_\chi}}\right). 
\end{equation}
Note that here we have also used the fact that $d(r) = d(r_1) d(r_2)$. 

We also know that
$$
\sum_{r_1} d(r_1) r_1^{-s} = F(p^{-s})^{-1},
$$
where the sum is taking over all powers of $p$. Moreover, we also have
$$
(\sum_i B_i p^{-is}) F(p^{-s})^{-1} = H(p^{-s}) F(p^{-s})^{-1} = (1 - \kappa p^{-s})^{-1}.
$$
Hence
$$
\sum_{i, r_1} B_i p^{-is} d(r_1) r_1^{-s} = \sum_{r_3} \kappa^{\mathrm{ord}_p r_3} r_3^{-s},
$$ 
where the $r_3$ varies over all powers of $p$. We also have the relation
$$
\kappa^{\mathrm{ord}_p r_3} =  \sum_{r_3 = p^i r_1} B_i d(r_1).
$$
Hence setting $r_3 = p^i r_1$, we see that the only terms that contribute to the sum in equation \eqref{eq: calc-step2} are those $p$-powers $r_3$ of the form $p^{j - j_\chi} r_4$ for some  $p$-power $r_4$. Also note that as $r_2$ varies over all positive integers prime to $p$, we get
$$
\sum_{r_2} \bar \chi (r_2) d (r_2) r_2^{-s}
= {G(s, \bar \chi, f)}.
$$
We remark that if $\chi$ is the trivial character of $\Z_p^\times$, then $G(s, \bar \chi, f)$ is just the $p$-deprived Asai $L$-function 
$G_p(s,f)^{-1} G(s, f)$, where $G_p(s,f)$ is the local Euler factor at $p$, since the function on $\Z$ induced by the trivial character 
$\chi$ is taken to vanish on $p\Z$.  
In any case, we can rewrite equation \eqref{eq: calc-step2} as
\begin{align*}
& \frac{p^{js -j_\chi}}{\kappa^j} G(\chi) G(s, \bar \chi, f) \sum_{r_3} \kappa^{\mathrm{ord}_p r_3} r_3^{-s} \bar \chi \left(\frac{r_3}{p^{j-j_\chi}}\right) \id_\Z \left(\frac{r_3}{p^{j-j_\chi}}\right) \\
&= \frac{p^{js -j_\chi}}{\kappa^j} G(\chi) G(s, \bar \chi, f) \sum_{r_4} \kappa^{j - j_\chi} \kappa^{\mathrm{ord}_p r_4} p^{-s(j -j_\chi)}  r_4^{-s} \bar \chi (r_4) \\
&=  \frac{p^{j_\chi (s - 1)}}{\kappa^{j_\chi}} G(\chi) G(s, \bar \chi, f) \sum_{r_4} \kappa^{\mathrm{ord}_p r_4} r_4^{-s} \bar \chi (r_4) \\
&= \frac{p^{j_\chi (s - 1)}}{\kappa^{j_\chi}} G(\chi) G(s, \bar \chi, f),
\end{align*}
since $\bar \chi (r_4) = 0$, unless $r_4 = 1$, since by convention all Dirichlet characters of $\Z_p^\times$, including the trivial character, 
are thought of as functions on $\Z$ by requiring that they vanish on $p\Z$. 
This simultaneously checks the distribution relations and establishes the relationship 
\begin{equation}
\label{int of chi wrt mutilde}
\int \chi\ d\tilde \mu_s =   \frac{p^{j_\chi (s - 1)}}{\kappa^{j_\chi}} G(\chi) G(s, \bar \chi, f)
\end{equation}
between these measures and twisted Asai $L$-values.

We remark here that these calculations hold only for $s \in \C$ where $G(s, f)$ is absolutely convergent. And this is known for all $s$ such that $\Re (s) > k+1$, in view of the Hecke bound $c(\mathfrak{l},f) = O(\mathrm{Nm}(\mathfrak{l})^{k/2})$, for all but finitely many primes $\mathfrak{l}$ of $F$. 

In order to construct a measure, we need to show that this is a bounded distribution (after possibly dividing by some periods). We now modify the distribution $\tilde \mu_s$ to construct the distribution
$$
\mu_s (a + p^j \Z_p) = \tilde \mu_s (a + p^j \Z_p) + \tilde \mu_s (-a + p^j \Z_p).
$$
The distribution relations for $\mu_s$ follows from those of $\tilde \mu_s$. Moreover
$$
\int \chi\ d\mu_s = \left\{ \begin{array}{cl} 2 \int \chi\ d\tilde \mu_s & \textrm{if $\chi$ is even,} \\
0 & \textrm{if $\chi$ is odd}. \end{array}\right.
$$

In the next section, we shall prove that the values of the distribution $\int \chi\ d\mu_s$ in \eqref{int of chi wrt mutilde}, for specific values of $s$, are rational, 
after dividing by some periods.
In Section~\ref{section bounded distributions}, we conjecture that these values are even $p$-adically bounded.
We then conjecture that these values satisfy the so called abstract Kummer congruences, and hence come from a measure.
 
\bigskip

\section{Rationality result for twisted Asai $L$-values}

Let $n = k-2$ and set $\uln = (n,n)$. For any $\calo_F$-algebra $A$, let $L(\uln, A)$ denote the set of polynomials in $4$ variables $(X, Y, \overline X, \overline Y)$ with coefficients in $A$, which are homogeneous of degree $n$ in $(X,Y)$ and homogeneous of degree $n$ in $(\overline X, \overline Y)$. We define an action of $\SL_2 (A)$ on this set by 
\begin{equation*}
\begin{split}
\gamma \cdot P(X, Y, \overline X, \overline Y) &=  P(\gamma^\iota (X, Y)^t, \bar{\gamma}^\iota (\overline X, \overline Y)^t) \\
&=  P(dX - bY, -cX + aY, \bar d \, \overline X -  \bar b \,\overline Y, - \bar c \overline X + \bar a \overline Y),
\end{split}
\end{equation*}
for $\gamma = \bmatrix a & b \\ c & d \endbmatrix \in \SL_2 (A)$ and where $\gamma^\iota = \det (\gamma) \gamma^{-1}$ is the adjoint matrix of $\gamma$. Let $\Gamma_0 (\fr n)$ and $\Gamma_1 (\fr n)$ be the usual congruence subgroups of $\SL_2(\calo_F)$ with respect to the ideal $\fr n$.
Let 
$$
\calh = \{ (z,t) \> | \>  z \in \C \textrm{ and } t \in \R \textrm{ with } t >0 \}
$$
be the hyperbolic upper half-space in $\R^3$. There is an action of $\SL_2 (\calo_F)$ on $\calh$ which is induced by identifying $\calh$ with $\SL_2 (\C)/\SU_2 (\C) = \GL_2(\C)/[\SU_2 (\C) \cdot \C^\times]$. The last identification is given by a transitive action of $\SL_2 (\C)$ on $\calh$ defined via
$g \mapsto g \cdot \epsilon$, for $g \in \SL_2(\C)$ and $\epsilon = (0,1) \in \calh$. We will let $\mathcal L (\uln, A)$ denote the system of local coefficients associated to $L(\uln, A)$. So $\call (\uln, A)$ is the sheaf of locally constant sections of the projection 
$$
\Gamma_0 (\fr n) \backslash (\calh \times L(\uln, A)) \longrightarrow \Gamma_0 (\fr n) \backslash \calh.
$$
Analogous to the Eichler-Shimura isomorphism for classical cusp forms, there are isomorphisms
$$
\delta_q : S_{\ulk} (\Gamma_0(\fr n)) \longrightarrow H^q_{cusp} (\Gamma_0 (\fr n) \backslash \calh, \call (\uln, \C)),
$$
for $q = 1,2$. Here the cohomology on the right is cuspidal cohomology with local coefficients. 
We take $\delta = \delta_1$ since we are interested in $1$-forms. 

Let us now describe the image of $\delta (f)$ under this map. Let $\gamma \in \SL_2 (\C)$ and $\ulz = (z,t) \in \calh$. After identifying $\ulz$ with the matrix $\bmatrix z & -t \\ t & \bar z \endbmatrix$, recall that the action of $\gamma$ on $\ulz$ is by
$$
\gamma \cdot \ulz = [\rho (a) \ulz + \rho (b) ][\rho (c) \ulz + \rho (d)]^{-1},
$$
where $\rho (\alpha) = \bmatrix \alpha & 0 \\ 0 & \bar \alpha \endbmatrix$. Define the automorphy factor $j(\gamma, \ulz) = \rho(c) \ulz + \rho(d) \in \GL_2 (\C)$. Let $L(2n+2, \C)$ denote the space of homogeneous polynomials of degree $2n+2$ in two variables $(S,T)$ and coefficients in $\C$. We will consider $L(2n+2, \C)$ with a left action of $\SL_2 (\C)$.  

Recall that $f$ is a function $\calh \to L (2n+2, \C)$ that satisfies the transformation property
$$
f(\gamma \ulz, (S,T)) =  f(\ulz, {^t} j(\gamma, \ulz) (S,T)^t),
$$
for $\gamma \in \Gamma_0 (\fr n)$. There is a related `cusp form' $F: \SL_2 (\C) \to L(2n+2, \C)$ on $\SL_2 (\C)$ which is defined by the formula
$$
f(\ulz, (S,T)^t) = F(g, {^t} j(g, \epsilon) (S,T)^t),
$$
where $g \in \SL_2(\C)$ is chosen such that $g \cdot \epsilon = \ulz$. 

By Clebsch-Gordon, there is an $\SU_2 (\C)$-equivariant homomorphism
$$
\Phi: L(2n+2, \C) \hookrightarrow L(\uln, \C) \otimes L(2, \C).
$$
Then $\delta(f)$ can explicitly be described as \cite[(13)]{Gh}
\begin{equation}
  \label{def of delta(f)}
\delta(f) (g) = g \cdot (\Phi \circ F(g)), \quad \forall g \in \SL_2 (\C).
\end{equation}
Note that here the action of $g$ on $L(\uln, \C)$ is as described above. But the action of $\SL_2 (\C)$ on $L(2, \C)$ is identified with the natural 
action of $\SL_2(\C)$ on $\Omega^1 (\calh) = \C dz \, \oplus \, \C dt \, \oplus \, \C d \bar z$ (see \cite[(6)]{Gh}). The identification of $L(2,\C)$ with $\Omega^1 (\calh)$ is given by sending $A^2 \mapsto dz, AB \mapsto -dt$ and $B^2 \mapsto -d\bar z$. With this identification, we view $\delta(f)$ as a $L(\uln, \C)$ valued differential form. It is also invariant under the action of $\SU_2 (\C)$ \cite[(14)]{Gh}, so it descends to a vector valued $1$-form on $\calh$. Moreover, 
if $\gamma \in \Gamma_0 (\fr n)$, then
$$
\delta (f) (\gamma \ulz) = \gamma \cdot (\delta (f) (\ulz)).
$$
So $\delta (f)$ descends to an element of $H^1_{cusp} (\Gamma_0(\fr n) \backslash \calh, \call (\uln, \C))$. 

We make this formula more explicit. Let $U$ and $V$ be auxiliary variables and define the following homogeneous polynomial of degree $2n+2$  by
$$
Q = \left( {{2n+2} \choose \alpha} (-1)^{2n+2-\alpha} U^\alpha V^{2n+2 - \alpha} \right)_{\alpha = 0, \dots, 2n+2}.
$$
Define $\psi (X,Y, \overline X, \overline Y, A, B) = [\psi_0 (X,Y, \overline X, \overline Y, A, B), \dots, \psi_{2n+2} (X,Y, \overline X, \overline Y, A, B)]^t$ by the formula
$$
(XV - YU)^n (\overline X U + \overline Y V)^n (AV - BU)^2 = Q \cdot \psi,
$$
where each $\psi_i$ is a polynomial that is homogeneous of degree $n$ in $(X,Y)$, homogeneous of degree $n$ in $(\overline X, \overline Y)$ and homogeneous of degree $2$ in $(A,B)$. 

Let $\ulz \in \calh$, then since $\SL_2 (\C)$ acts transitively on $\calh$, there is a $g \in \SL_2 (\C)$ such that $\ulz = g \cdot \epsilon$ where $\epsilon = (0,1) \in \calh$. Let $F^\alpha : \SL_2 (\C) \to \C$, for $\alpha = 0, \dots, 2n+2$, 
be the components of the function $F : \SL_2 (\C) \to L(2n+2 , \C)$. Then 
$$
(\Phi \circ F) (g) = [F^0 (g), \dots, F^{2n+2} (g)] \cdot \psi (X,Y, \overline X, \overline Y, A, B)
$$
and
$$
\delta (f) (\ulz) = g \cdot (\Phi \circ F (g))  = [F^0 (g), \dots, F^{2n+2} (g)] \cdot \psi (g^\iota (X,Y)^t, \bar{g}^\iota (\overline X, \overline Y)^t, {^t} j(g^{-1}, \epsilon)^{-1} (A, B)^t),
$$
where $A^2, AB, B^2$ are replaced by $dz, -dt, -d\bar z$.

Now, let $\beta \in F \subset \C$ and let $T_\beta$ denote the translation map $\calh \to \calh$ given by sending $\ulz = (z,t) \mapsto (z+\beta, t)$. When we view $\calh$ as a quotient space of $\SL_2 (\C)$, this map is induced by sending the coset $g \SU_2 (\C) \mapsto \gamma_\beta g \SU_2(\C)$, where 
$$
\gamma_\beta := \bmatrix 1 & \beta \\ 0 & 1 \endbmatrix.
$$
Let $\Gamma_0^\beta (\fr n) := \gamma_\beta^{-1} \Gamma_0 (\fr n) \gamma_\beta \subset \SL_2 (F)$. Notice that if $\Gamma_0^\beta (\fr n) g^\prime \SU_2 (\C) = \Gamma_0^\beta (\fr n) g \SU_2 (\C)$, then $\Gamma_0 (\fr n) \gamma_\beta g^\prime \SU_2 (\C) = \Gamma_0 (\fr n) \gamma_\beta g \SU_2 (\C)$. 
Thus the translation map $T_\beta$ induces a well-defined map
$$
\Gamma_0^\beta (\fr n) \backslash \calh \stackrel{T_\beta}{\longrightarrow} \Gamma_0 (\fr n) \backslash \calh,
$$
which we again denote by $T_\beta$.

We now recall some basic facts about functoriality of cohomology with local coefficients. For $i = 1,2$, let $X_i$ be topological spaces 
with universal covers $\tilde X_i$ and fundamental groups $\Gamma_i$ (after fixing some base points). Let $M_i$ be local coefficient systems on $X_i$, 
i.e., each $M_i$ is an abelian group with an action of the fundamental group $\Gamma_i$.  Let $\phi : X_1 \to X_2$ be a map between the spaces, 
it induces a map $\phi_*: \Gamma_1 \to \Gamma_2$ on the fundamental groups. A map between the coefficient systems $\tilde \phi: M_2 \to M_1$ is 
said to be compatible with $\phi$ if it satisfies
$$
\gamma_1 \tilde \phi (m_2) = \tilde \phi (\phi_* (\gamma_1) m_2), \quad \forall m_2 \in M_2 \textrm{ and } \gamma_1 \in \Gamma_1.
$$
In other words, $\tilde \phi$ must be a map between representations when $M_2$ is viewed as a representation of $\Gamma_1$ via the map $\phi_*$. For any compatible pair $(\phi, \tilde \phi)$, there exists an induced map 
$$
\phi^*: H^q (X_2, M_2) \to H^q (X_1, M_1)
$$
at the level of cohomology. This map is constructed as follows. Let $S_* (\tilde X_i)$ denote the singular complex of the universal covers. There is a natural action on the right by $\Gamma_i$ via deck transformations. Given a singular $q$-simplex $\sigma : \Delta^q \to \tilde X_i$ and $g \in \Gamma_i$, we convert this right action into a left action by setting $ g \cdot \sigma = \sigma \cdot g^{-1}$. The cohomology groups  with local coefficients are given by the homology of the complex $\Hom_{\Z \Gamma_i} (S_* (\tilde X_i), M_i)$. The map $\phi^*$ is induced by the following map on the complexes (which we again denote by $\phi^*$)
$$
\phi^*: \Hom_{\Z \Gamma_2} (S_q (\tilde X_2), M_2) \to \Hom_{\Z \Gamma_1} (S_q (\tilde X_1), M_1).
$$
Given a cochain $C \in \Hom_{\Z \Gamma_2} (S_q (\tilde X_2), M_2)$ and $\tau : \Delta^q \to \tilde X_1 \in S_q (\tilde X_1)$, define 
\begin{equation}
  \label{def of phi star}
\phi^* (C) (\tau) = \tilde \phi ( C(\phi \circ \tau)),
\end{equation}
where we continue to denote by $\phi : \tilde X_1 \rightarrow \tilde X_2$ the unique lift of $\phi : X_1 \rightarrow X_2$ to the universal covers.
This construction is independent of the base points chosen in the beginning. 

Let us now apply this to our situation with $X_1 = \Gamma_0^\beta (\fr n) \backslash \calh$ and  $ X_2 = \Gamma_0 (\fr n) \backslash \calh$. In this case, $\Gamma_1 = \Gamma_0^\beta (\fr n)$ and $\Gamma_2 = \Gamma_0 (\fr n)$ and they act on $M_1 = M_2 = L(\uln, \C)$ since they  are subgroups of $\SL_2 (F)$. We take the map $\phi$ to be the translation map $T_\beta$ which induces, at the level of fundamental groups, the map $(T_\beta)_* : \Gamma_0^\beta (\fr n) \to \Gamma_0 (\fr n)$ which sends $\gamma_\beta^{-1} \gamma \gamma_\beta \mapsto \gamma$. It is an easy check that the map $\tilde T_\beta : M_2 \to M_1$ sending $P \mapsto \gamma_\beta^{-1} P$, for $P \in L(\uln, \C)$, is compatible with $T_\beta$. By the discussion above this induces a map at the level of cohomology
$$
T_\beta^* : H^q (\Gamma_0 (\fr n) \backslash \calh, \call (\uln, \C)) \to H^q (\Gamma_0^\beta (\fr n) \backslash \calh, \call (\uln, \C)).
$$
When $q=1$, what is the image of the element $\delta (f) \in H^1 (\Gamma_0 (\fr n) \backslash \calh, \call (\uln, \C))$? After translating the above map $T_\beta^*$ in terms of vector valued differential forms, we get that
\begin{equation*}
\begin{split}
T_\beta^* (\delta (f)) (\ulz) & \stackrel{\eqref{def of phi star}}{=} \gamma_\beta^{-1} \delta (f)(\gamma_\beta \ulz) \stackrel{\eqref{def of delta(f)}}{=} \gamma_\beta^{-1} \gamma_\beta g \cdot (\Phi \circ F(\gamma_\beta g)) = g \cdot (\Phi \circ F (\gamma_\beta g)) \\
&= [F^0 (\gamma_\beta g), \dots, F^{2n+2} (\gamma_\beta g)] \cdot \psi (g^\iota (X,Y)^t, \bar{g}^\iota (\overline X, \overline Y)^t, {^t} j(g^{-1}, \epsilon)^{-1} (A, B)^t). 
\end{split}
\end{equation*}
Here $\ulz \in \calh$ and we take $g \in \SL_2 (\C)$ such that $g \ulz = \epsilon$, and $A^2, AB, B^2$ are to be replaced by $dz, -dt, -d\bar z$.

Following \cite[\S 5.2]{Gh}, we now want to compute the restriction $T_\beta^*(\delta(f))|_{\H}$ where $\H = \{x + it \mid x, t \in \R\ \mathrm{and}\ t>0 \}$ is the usual upper half-plane which is embedded into the hyperbolic $3$-space $\calh$ as
$$
x + it \mapsto \bmatrix x & -t \\ t & x \endbmatrix.
$$
As in {\it loc.\,cit.}, we make the following two simplifications. Firstly, since we wish to compute this differential form on $\H$, we set $dz = d \bar z$ in our computations. Secondly, we only need to calculate the differential form $\bmatrix 1 & -x \\ 0 & 1 \endbmatrix \cdot \big(T_\beta^*(\delta(f))|_{\H}\big)$, so we set $x=0$ in $\psi$ and only calculate the modified differential form which we denote by $\widetilde{T_\beta^*(\delta(f))|_{\H}}$. 
Note that the components $\psi_\alpha$ of $\psi$, for $\alpha = 0, \dots, 2n+2$, are given by 
$$
\psi_\alpha (X,Y, \overline X, \overline Y, A, B) = (-1)^\alpha \frac{A^2 c_\alpha - 2AB c_{\alpha-1} + B^2 c_{\alpha-2}}{{2n+2 \choose \alpha}},
$$
where 
$$
c_\alpha (X, Y, \overline X, \overline Y) = \sum^n_{\stackrel{j,k = 0}{n = \alpha + j- k}}
(-1)^k {n \choose j}{n \choose k} X^{n-k} Y^k \overline{X}^{n-j} \overline{Y}^{j}.
$$

For $x, t \in \R$ and $t >0$, let $g = \frac{1}{\sqrt t} \bmatrix t & x \\ 0 & 1 \endbmatrix$.  Then $g \cdot \epsilon = (x, t) \in \H \subset \calh$ and $g^\iota = \bar{g}^\iota = \frac{1}{\sqrt t} \bmatrix 1 & -x \\ 0 & t \endbmatrix$. Moreover, $j (g^{-1}, \epsilon)^{-1} = j(g, \epsilon) = \frac{1}{\sqrt t} \bmatrix 1 & 0 \\ 0 & 1 \endbmatrix$. Let $f^\alpha$, for $\alpha = 0, \dots, 2n+2$, be the components of $f : \calh \to L(2n+2, \C)$. The precise relationship between $f^\alpha$ and $F^\alpha$ is given by 
$$
f^\alpha (\ulz) = \frac{1}{\sqrt{t}^{2n+2}} F^\alpha (g).
$$
Note that if $\ulz = (z,t)$, then $T_\beta(\ulz) = (z+\beta, t)$ does not affect the $t$ coordinate. Hence
$$
f^\alpha (T_\beta \ulz) = \frac{1}{\sqrt{t}^{2n+2}} F^\alpha (\gamma_\beta g).
$$
Using this and the pullback formula, we get
\begin{equation*}
\begin{split}
\widetilde{T_\beta^*(\delta(f))|_{\H}}  &= \sum_\alpha^{2n+2} \sqrt{t}^{2n+2} f^\alpha (T_\beta \ulz) \psi_\alpha \left(\frac{1}{\sqrt t} X, {\sqrt t}\ Y, \frac{1}{\sqrt t} \overline X, {\sqrt t}\ \overline Y, \frac{1}{\sqrt t} A, \frac{1}{\sqrt t} B\right) \\
&= \sum_\alpha^{2n+2} f^\alpha (T_\beta \ulz) \psi_\alpha (X, tY, \overline X, t \overline Y, A, B),
\end{split}
\end{equation*}
where we  replace $(A^2, AB, B^2) $ by $(dx, -dt, -dx)$. We have now constructed an element  $\widetilde{T_\beta^*(\delta(f))|_{\H}}  \in H^1_{cusp} (\Gamma_0^\beta (N) \backslash \H, \call (\uln, \C))$ where $\Gamma_0^\beta (N) := \Gamma_0^\beta (\fr n) \cap \Gamma_0 (N) = \Gamma_0^\beta (\fr n) \cap \SL_2(\Z)$, since
in the latter matrix group, the lower left entries are divisible by $N$.  
As in \cite[see below Lemma 2]{Gh}, we have a decomposition of this cohomology group as
\begin{equation}
  \label{Clebsch Gordan}
H^1_{cusp} (\Gamma_0^\beta (N) \backslash \H, \call (\uln, \C)) \stackrel{\sim}{\longrightarrow}  \bigoplus_{m=0}^n H^1_{cusp} (\Gamma_0^\beta (N) \backslash \H, \call (2n-2m, \C)).
\end{equation}
We will call the projection of $\widetilde{T_\beta^*(\delta(f))|_{\H}}$ into the $m$-th component by $\widetilde{T_\beta^*(\delta_{2n-2m}(f))}$, slightly 
abusing notation since the subscript $2n-2m$ should technically be outside the parentheses. For each $m$, define
\begin{equation}
  \label{g alpha}
g^\alpha (\ulz) = \left\{ \begin{array}{ll} 
\frac{f^{\alpha} (\ulz) + (-1)^{n+1-\alpha + m} f^{2n+2 - \alpha} (\ulz)}{{2n+2 \choose \alpha} } & \mathrm{if}\ \alpha = 0, 1, \dots, n, \\
\frac{f^{n+1} (\ulz)}{{2n+2 \choose n+1} } & \mathrm{if}\ \alpha = n+1.
\end{array} \right.
\end{equation}
Then, we have
\begin{equation}
  \label{def of T beta delta 2n-2m}
\widetilde{T_\beta^*(\delta_{2n-2m}(f))}  (x,t) = \sum_{l=0}^{2n-2m} (A_l dx + 2 B_l dt) t^{2n-m-l} X^l Y^{2n-2m-l},
\end{equation}
where
\begin{gather*}
A_l = \sum_{\alpha = 0}^{n+1} (-1)^\alpha g^\alpha (T_\beta (x,t)) a(m,l, \alpha), \\
B_l = \sum_{\alpha = 0}^{n+1} (-1)^\alpha g^\alpha (T_\beta (x,t)) b(m,l, \alpha),
\end{gather*}
with $a(m,l, \alpha)$ and $b(m,l,\alpha)$ the integers defined at the end of \cite[\S 5]{Gh}.

For any $n \geq 0$ and any $\Z[1/n!]$-algebra $A$, there is an $\SL_2 (\Z)$-equivariant pairing \cite[Lemma 4]{Gh}
$$
\langle \ , \ \rangle: L(n, A) \otimes L(n, A) \to A,
$$
which induces by Poincare duality a pairing 
$$
\langle \ , \ \rangle: H^1_{c} (\Gamma_0^\beta (N) \backslash \H, \call (n, A)) \otimes H^1 (\Gamma_0^\beta (N) \backslash \H, \call (n, A)) \to H^2_{c} (\Gamma_0^\beta (N) \backslash \H, A) \to A,
$$
where the last map $H^2_{c} (\Gamma_0^\beta (N) \backslash \H, A) \to A$ is given by integrating a compactly supported $2$-form on a 
fundamental domain $[\Gamma_0^\beta (N) \backslash \H]$ of $\Gamma_0^\beta (N) \backslash \H$. We will use this pairing when $A = \C$, $A = E$ is a $p$-adic number field
with $p > n$, and with $A = \calo_E$, its ring of integers.  When $A = \C$, the pairing can be extended to  
$$\langle \ , \ \rangle: H^1_{cusp} (\Gamma_0^\beta (N) \backslash \H, \call (n, \C)) \otimes H^1 (\Gamma_0^\beta (N) \backslash \H, \call (n, \C)) \to H^2_{cusp} (\Gamma_0^\beta (N) \backslash \H, \C) \to \C \cup \{\infty\}.
$$

For each $m$, there is an Eisenstein differential form
$E^\beta_{2n-2m+2}$ for $\Gamma_0^\beta(N)$ given by 
\begin{equation}
  \label{Eis beta}
  E^\beta_{2n-2m+2}(s,z) = \sum_{\gamma \, \in \, {\Gamma^\beta_0(N)}_\infty \backslash \Gamma_0^\beta(N)}
            \>  \gamma^{-1} \cdot \gamma^* (\omega \, y^s),
\end{equation}
where $\omega = (X-zY)^{2n-2m}dz$. One may check that
\begin{eqnarray}
  \label{Eis beta 2}
  E^\beta_{2n-2m+2}(s,z) = \sum_{\gamma \, = \, \left(\begin{smallmatrix} a & b \\ c & d \end{smallmatrix}\right) \, \in \, {\Gamma^\beta_0(N)}_\infty  \backslash
                                                            \Gamma^\beta_0(N)}
             \> \frac{1}{(cz+d)^{2n-2m+2} |cz+d|^{2s}}
             \cdot y^s \, \omega.
\end{eqnarray}
We view $E^\beta_{2n-2m+2}$ as an element of $H^1 (\Gamma_0^\beta (N) \backslash \H, \call (n, \C))$. We now wish to evaluate
$$
\langle T_\beta^*(\delta_{2n-2m}(f)), E^\beta_{2n-2m+2} \rangle,
$$
following \cite[\S 6.3]{Gh}.
We have
\begin{equation*}
\begin{split}
\langle T_\beta^*(\delta_{2n-2m}(f)), E^\beta_{2n-2m+2} \rangle &= \int_{[\Gamma_0^\beta (N) \backslash \H]}  \langle T_\beta^*(\delta_{2n-2m}(f)) (x,t), E^\beta_{2n-2m+2} (x,t) \rangle \\
&= \int_{[\Gamma_0^\beta (N) \backslash \H]}  \langle \widetilde{T_\beta^*(\delta_{2n-2m}(f))}, \widetilde{E^\beta_{2n-2m+2}} \rangle,
\end{split}
\end{equation*}
where 
the $\widetilde{\quad}$ indicates that we have twisted the differential forms by the action of the matrix $\bmatrix 1 & -x \\ 0 & 1 \endbmatrix$.  Using a standard unwinding argument, the last integral becomes
$$
\int_0^\infty \int_0^1 \langle \widetilde{T_\beta^*(\delta_{2n-2m}(f))}, \widetilde{\omega} t^s \rangle,
$$
where $\widetilde \omega = (X-itY)^{2n-2m} dz$. Using the expression \eqref{def of T beta delta 2n-2m} for $\widetilde{T_\beta^*(\delta_{2n-2m}(f))}  (x,t)$ and the definition of the pairing, we have
\begin{equation*}
\begin{split}
\int_0^\infty \int_0^1 \langle \widetilde{T_\beta^*(\delta_{2n-2m}(f))}, \widetilde{\omega} t^s \rangle =  & \int_0^\infty \int_0^1 \sum_{l=0}^{2n-2m} i^{l+1} A_l t^{2n-m+s} \ dxdt \\ 
& \quad - 2 \int_0^\infty \int_0^1 \sum_{l=0}^{2n-2m} i^{l} B_l t^{2n-m+s} \ dxdt.
\end{split}
\end{equation*}
We denote the first integral by $I_1$ and the second integral by $I_2$. We now compute $I_1$ using the definition of $A_l$ as
$$
I_1 = \sum_{l=0}^{2n-2m} i^{l+1} \sum_{\alpha = 0}^{n+1} (-1)^\alpha a(m,l,\alpha) \int_0^\infty \int_0^1 g^\alpha (T_\beta (x,t)) t^{2n-m+s}\ dx dt.
$$
Using the Fourier expansion for the $\alpha$-th component of $f$, see \cite[(7)]{Gh} with $a_1 = 1$,  we get
$$
f^\alpha (T_\beta (x,t)) = t { 2n+2 \choose \alpha} \left[ \sum_{\xi \in F^\times} c(\xi d) \left(\frac{\xi}{i |\xi|}\right)^{n+1-\alpha} K_{\alpha - n - 1} (4\pi t |\xi|) e_F (\xi (x+\beta)) \right],
$$
where $e_F (w) = e^{2 \pi i \Tr_{F/\Q} (w)} $. Using \eqref{g alpha} and plugging this into the expression for $I_1$, we get
\begin{equation*}
\begin{split}
I_1 = & \sum_{l=0}^{2n-2m}  i^{l+1}  \sum_{\alpha = 0}^n (-1)^\alpha a(m,l,\alpha) \int_0^\infty \sum_{\xi \in F^\times} c(\xi d) t^{2n-m+1+s} \\
& \left( \left(\frac{\xi}{i |\xi|}\right)^{n+1-\alpha} K_{\alpha - n - 1} (4\pi t |\xi|) 
+ (-1)^{n+m+1-\alpha} \left(\frac{\xi}{i |\xi|}\right)^{\alpha - n - 1} K_{ n + 1 - \alpha} (4\pi t |\xi|) \right) dt  \\
& \int_0^1 e_F (\xi (x+\beta)) dx + \sum_{l=0}^{2n-2m} i^{l+1} (-1)^{n+1} a(m,l, n+1) \int_0^\infty \sum_{\xi \in F^\times} c(\xi d) t^{2n+1-m+s} \\
& K_0 (4\pi t |\xi|) dt \int_0^1 e_F (\xi (x+\beta)) dx.
\end{split}
\end{equation*}
The only terms $c(\xi d)$ that survive are when $\xi = \frac{r}{\sqrt{-D}}$, for some $0 \neq r \in \Z$, and in this case $\int_0^1 e_F (\xi x) dx = 1$. 
\begin{equation*}
\begin{split}
I_1 = & \sum_{l=0}^{2n-2m}  i^{l+1} \sum_{\alpha = 0}^n  (-1)^\alpha a(m,l,\alpha) \sum_{r \neq 0} e_F (r\beta/\sqrt{-D}) c(r) \left(\frac{-r}{|r|}\right)^{n+1-\alpha}
\\ 
& \int_0^\infty t^{2n+1-m+s} \left[K_{\alpha - n - 1} \left(\frac{4\pi t |r|}{\sqrt{D}}\right) + (-1)^{n+m+1-\alpha} K_{ n + 1 -\alpha} \left(\frac{4\pi t |r|}{\sqrt{D}}\right) \right] dt \\
& + \sum_{l=0}^{2n-2m} i^{l+1} (-1)^{n+1} a(m,l,n+1) \sum_{r \neq 0} e_F (r\beta/\sqrt{-D}) c(r) \int_0^\infty t^{2n+1 - m +s} K_0 \left(\frac{4\pi t |r|}{\sqrt{D}}\right) dt.
\end{split}
\end{equation*}
The Bessel functions have the property \cite[Lemma 7]{Gh}
$$
\int_0^\infty K_\nu (at) t^{\mu -1} dt = 2^{\mu-2} a^{-\mu} \Gamma\left(\frac{\mu+\nu}{2}\right)  \Gamma\left(\frac{\mu-\nu}{2}\right). 
$$
This implies that the two Bessel functions in the sum above will cancel each other unless $ \alpha \equiv  n+1+m \mod (2)$. Setting $s^\prime = 2n+2 -m +s$, we have
\begin{equation*}
\begin{split}
I_1 = & \> \frac{(-1)^{n+1} \sqrt{D}^{s^\prime}}{2 (2\pi)^{2n+2 -m+s}} \sum_{l=0}^{2n-2m} i^{l+1} \sum_{\stackrel{\alpha=0}{\alpha \equiv n+1+m\ (2)}}^n (-1)^m a(m,l,\alpha)  \sum_{0 \neq r \in \Z} e_F (r\beta/\sqrt{-D}) c(r) \\
& \left(\frac{-r}{|r|}\right)^{n+1-\alpha} \frac{1}{|r|^{s'}} \Gamma\left(\frac{n+1-m+\alpha+s}{2}\right) \Gamma\left(\frac{3n+3-m-\alpha+s}{2}\right) \\
& + \frac{(-1)^{n+1} \sqrt{D}^{s^\prime}}{4 (2\pi)^{2n+2 -m+s}} \sum_{l=0}^{2n-2m} i^{l+1} a(m,l,n+1) \sum_{0 \neq r \in \Z} e_F (r\beta/\sqrt{-D}) c(r) \frac{1}{|r|^{s'}} \\
& \Gamma\left(\frac{2n+2 -m +s}{2}\right)^2.
\end{split}
\end{equation*}
We will take $\beta = \frac{b\sqrt{-D}}{2}$ for some rational number $b$. Then the term $e_F (r\beta/\sqrt{-D}) = e^{2\pi i rb}$.

Now we break the sum over $r$ into a sum over positive integers and a sum over negative integers.  The term $\left(\frac{-r}{|r|}\right)^{n+1-\alpha}$ equals $(-1)^m$ when $r$ is positive and is $1$ when $r$ is negative. The second sum over $r$ does not have such a term, so we assume that $m$ is even in order to be able to put these terms together into a single term. The terms $c(r)$ and $|r|^{s'}$ are obviously independent of the sign of $r$.  
So finally, we have
\begin{equation}
\begin{split}
I_1 = & \> \frac{(-1)^{n+1} \sqrt{D}^{s^\prime}}{2 (2\pi)^{2n+2 -m+s}} \sum_{r = 1}^\infty \left(e^{2\pi i r b} + e^{-2 \pi i r b}\right) \frac{c(r)}{r^{s^\prime}} \sum_{l = 0}^{2n-2m} i^{l+1} \\
& \sum_{\stackrel{\alpha=0}{\alpha \equiv n+1+m\ (2)}}^{n+1} a(m,l,\alpha) \Gamma\left(\frac{n+1-m+\alpha+s}{2}\right) \Gamma\left(\frac{3n+3-m-\alpha+s}{2}\right),
\end{split}
\end{equation}
where there is an extra factor of $\frac{1}{2}$ in the $\alpha = n+1$ term, which we will adjust for. 
By a similar computation, $I_2$ will also have an expression in terms of $b(m,l,\alpha)$. Putting together these two expressions, we get that
$$
\langle T_\beta^*(\delta_{2n-2m}(f)), E^\beta_{2n-2m+2} (s) \rangle = \frac{\sqrt{D}^{s^\prime}}{(2\pi)^{2n+2-m+s}}  \sum_{r = 1}^\infty \left(e^{2\pi i r b} + e^{-2 \pi i r b}\right) \frac{c(r)}{r^{s^\prime}} \cdot G'_{\infty} (s,f),
$$
where we collect all the combinations of Gamma factors appearing in both $I_1$ and $I_2$ and denote it by $G'_{\infty} (s,f)$. 

Now let $\chi : (\Z/p^j \Z)^\times \to \C^\times$ be a primitive character (so $j = j_\chi$) and recall that 
$$
G(s,\bar \chi, f) = L_N (2s-2k+2, \bar \chi^2)  \sum_{r=1}^\infty \frac{ c(r) \bar \chi(r)}{r^s}.
$$
Substituting the formula
$$
\bar \chi (r) = \frac{1}{G(\chi)} \sum_{a \mod p^j}  \chi (a) e^{2\pi i r a/ p^j}
$$
in the above equation, we get
$$
G(s^\prime, \bar \chi, f) = L_N (2n-2m+2+2s, \bar \chi^2) \frac{1}{G( \chi)} \sum_{a \mod p^j}  \chi (a)  \sum_{r=1}^\infty \frac{c(r)}{r^{s^\prime}} e^{2\pi i r a/ p^j}.
$$

Now assume that $\chi$ is an even character, i.e., $ \chi(-1) = 1$. Then grouping together the terms coming from $a$ and $-a$, we get
$$
G(s^\prime, \bar \chi, f) = L_N (2n-2m+2+2s, \bar \chi^2) \frac{1}{G(\chi)} \sum_{a \in R}  \chi (a)  \sum_{r=1}^\infty \frac{c(r)}{r^{s^\prime}} (e^{2\pi i r a/p^j} +  e^{-2\pi i r a/p^j}),
$$
where $R$ is half of the representatives modulo $p^j$ such that if $a \in R$, then $-a \not \in R$.  We now write $G(s^\prime, \chi, f)$ in terms of the inner product considered earlier
\begin{multline}
G(\chi) G(s^{\prime}, \bar \chi, f) = \frac{(2\pi)^{2n+2-m+s}}{G'_{\infty} (s,f) \sqrt{D}^{s^\prime}}  L_N (2n-2m+2+2s, \bar \chi^2) \\ \sum_{a \in R}  \chi (a) \langle T_{\beta}^*(\delta_{2n-2m}(f)), E^\beta_{2n-2m+2} (s) \rangle,
\end{multline}
with $\beta = a\sqrt{-D}/2p^j$. 

Let $G_{\infty} (s,f) = G'_\infty (s,f) \Gamma(s+ 2n-2m+2)$. 
Dividing both sides of the above equation by the period $G(\bar \chi^2)  (2\pi)^{2n-2m+2}$, we obtain
\begin{multline*}
\frac{G(\chi) G(s^{\prime}, \bar \chi, f)}{G(\bar \chi^2)(2\pi)^{2n-2m+2}} = \frac{(2\pi)^{2n+2-m+s}}{G_{\infty} (s,f) \sqrt{D}^{s^\prime}} \cdot \frac{L_N (2n-2m+2+2s, \bar \chi^2)}{(2\pi)^{2n-2m+2} G(\bar \chi^2)} \cdot \Gamma(s+2n-2m+2) \\
\cdot \sum_{a \in R}  \chi (a) \langle T_{\beta}^*(\delta_{2n-2m}(f)), E^\beta_{2n-2m+2} (s) \rangle ,
\end{multline*}
We evaluate this expression at $s=0$. Note that $G_\infty(0,f) \neq 0$, by \cite[\S 6.4]{Gh} using some special arguments, and by Lanphier-Skogman and Ochiai
\cite{sl} as a consequence of their proof of \cite[Conjecture 1]{Gh}. The special value $L_N (2n-2m+2, \bar \chi^2)$ becomes rational after 
dividing by the period 
$G(\bar \chi^2)  (2\pi)^{2n-2m+2}$. We denote this ratio by  $L^\circ (2n-2m+2, \bar \chi^2)$. 
We get
\begin{multline}\label{integral-expression}
\frac{G(\chi) G(2n-m+2, \bar \chi, f)}{G(\bar \chi^2) \Omega_\infty} =  L^\circ (2n-2m+2, \bar \chi^2) \sum_{a \in R} \chi (a) \langle T_{\beta}^*(\delta_{2n-2m}(f)), E^\beta_{2n-2m+2} (0) \rangle.
\end{multline}
Here $\Omega_\infty$ is defined as 
$$
\Omega_\infty = \frac{(2\pi)^{4n-3m+4} \Gamma(2n-2m+2)}{G_\infty (0,f) \sqrt{D}^{2n-m+2}}.
$$

We now conclude rationality properties of the special values $G(2n-2m+2, \bar \chi, f)$ from equation \eqref{integral-expression}. 
Choose a period $\Omega(f)$ such that after dividing by this period, the differential form 
$$
\delta^\circ (f):= \frac{\delta(f)}{\Omega(f)} \in H^1_{cusp} (\Gamma_0 (\fr n) \backslash \calh, \call (\uln, E))
$$
 takes rational values. Here $E$ is a sufficiently large $p$-adic field, containing the field of rationality of the form $f$, which we also view as a subfield of $\C$ after fixing an isomorphism between $\C$ and $\qpbar$.  Then $$T_\beta^* \delta(f) |_{\H} = \Omega(f)  \cdot T_\beta^* \delta^\circ(f) |_{\H},$$ noting that if $\sqrt{-D} \in E$, which we assume, 
then the image 
$T_\beta^* \delta^\circ(f) |_{\H}$ of $\delta^\circ(f)$ under the map
$$
T_\beta^*|_{\H}: H^1_{cusp} (\Gamma_0 (\fr n)\backslash \calh, \call (\uln, E)) \to H^1_{cusp} (\Gamma^\beta_0 (\fr n)\backslash \calh, \call (\uln, E)) 
                              \to H^1_{cusp} (\Gamma^\beta_0 (N) \backslash \H, \call (\uln, E)), 
$$
is also rational. Since Clebsch-Gordan preserves rationality, for $0 \leq m \leq n$, we obtain that 
$$T_\beta^*(\delta_{2n-2m}(f)) = \Omega(f)  \cdot T_\beta^*(\delta_{2n-2m}^\circ(f)),$$
where $T_\beta^*(\delta_{2n-2m}^\circ(f)) \in H^1_{cusp}(\Gamma^\beta_0(N)\backslash \H, \call (2n-2m, E))$ is also rational.

The rational cuspidal class $T_\beta^*(\delta_{2n-2m}^\circ(f))$ is cohomologous to a compactly supported rational class which has the same value
when paired with $E^\beta_{2n-2m+2} (0)$ (see the proof of \cite[Theorem 1]{Gh}).
Since the differential form $E^\beta_{2n-2m+2} (0)$ coming from the Eisenstein series is $E$-rational, at least when $m \neq n$ (see 
Proposition~\ref{prop rational} in Section~\ref{section Eisenstein} below), and the pairing between
compactly supported rational classes and such classes preserves $E$-rationality, the following theorem follows 
from \eqref{integral-expression},  if $E$ contains the field of rationality of $\chi$, which we again assume. 

\begin{theorem}[Rationality result for twisted Asai $L$-values] Let $E$ be a sufficiently large $p$-adic number field with $p \nmid 2ND$. 
Let $0 \leq m < n$ be even and $\chi$ be even. Then 
\begin{equation*}
\frac{G(\chi) G(2n-m+2, \bar \chi, f)}{ G(\bar \chi^2) \Omega(f) \Omega_\infty}  \in E.
\end{equation*}
\end{theorem}

\noindent This result matches with \cite[Theorem 1]{Gh} when $\chi$ is trivial.  In that theorem it was assumed that the finite 
part of the central character of $f$ is non-trivial primarily to deal with the rationality of the Eisenstein series when $m = n$. 
In this paper, we have assumed (for simplicity) that the finite part of the central 
character of $f$ is trivial. We could still probably include the case $m = n$ in the theorem above, by using the rationality of 
the Eisenstein series $E^\beta_{2}(0,z) - pE^\beta_2(0,pz)$ instead (see \cite[Remark 2]{Gh}).

\section{Rationality of Eisenstein cohomology classes}
  \label{section Eisenstein}

We start by recalling the following result that goes back to Harder \cite{Ha4}, \cite{Ha}. See also \cite[\S 10]{Hi}.

\begin{lemma} 
  \label{lemma rational}
  Eisenstein cohomology classes corresponding to Eisenstein series whose constant terms at every cusp are rational are rational cohomology 
  classes. 
\end{lemma}

\begin{proof}
We use notation in this proof that is independent of the rest of the paper.   
Let $\Gamma \subset  \SL_{2}(\mathbb{Z})$  be a congruence subgroup and $\call(n, \mathbb{C})$ denote the sheaf of locally constant 
sections of $\pi: \Gamma \backslash (\mathbb{H} \times L(n,\mathbb{C})) \rightarrow \Gamma  \backslash \mathbb{H}$. Consider
the restriction map to boundary cohomology given by 
   \begin{align*}
  R : H^{1}(\Gamma \backslash \mathbb{H}, \call(n, \mathbb{C})) \rightarrow H^{1}_\partial (\Gamma \backslash \mathbb{H}, \call(n, \mathbb{C}))  := \bigoplus_{\xi}  H^{1} (\Gamma_{\xi} \backslash \mathbb{H}, \call(n, \mathbb{C})),
   \end{align*}
where $\xi$ varies through the cusps of $\Gamma$. 
We know that $$H^{1}(\Gamma \backslash \mathbb{H}, \call(n, \mathbb{C})) = H^{1}_{\text{cusp}}(\Gamma \backslash \mathbb{H}, \call(n, \mathbb{C})) \oplus H^{1}_{\text{Eis}}(\Gamma \backslash \mathbb{H}, \call(n, \mathbb{C})),$$ 
where $H^{1}_{\text{cusp}}$ and $H^{1}_{\text{Eis}}$ are the cuspidal and Eisenstein part of cohomology respectively. The restriction of  $R$  to $H^{1}_{\text{Eis}}$ is an isomorphism 
  \begin{align*}
  R: H^{1}_{\text{Eis}}(\Gamma \backslash \mathbb{H},\call(n, \mathbb{C})) & \rightarrow  \bigoplus_{\xi}  H^{1} (\Gamma_{\xi} \backslash \mathbb{H}, \call(n, \mathbb{C})) \\
  \omega  & \mapsto c_{\xi}(0,\omega),
  \end{align*} 
  where $c_{\xi}(0,\omega)$ is the differential form corresponding to the 
  ``constant term" in the Fourier expansion at the cusp $\xi$ of the differential form $\omega$ corresponding to the underlying Eisenstein series.
  Clearly $R$ preserves the rational structures on both sides. 
The following fact is due to Harder. 

\vspace{.2cm}

{\noindent \textbf{Fact}:} There exists a section 
  $$M: \oplus_{\xi}  H^{1} (\Gamma_{\xi} \backslash \mathbb{H}, \call(n, \mathbb{C}))  \rightarrow H^{1}(\Gamma \backslash \mathbb{H}, \call(n, \mathbb{C}))$$ 
of $R$ preserving rational structures on both sides. 

\vspace{.2cm}

\noindent Now let $\omega \in H^{1}_{\text{Eis}}(\Gamma \backslash \mathbb{H},\call(n, \mathbb{C}))$ be  such that $R(\omega)$ is rational. 
Then $M ( R (\omega) )$ is rational and $R(M(R(\omega))) = R(\omega)$. Since $R$ is an isomorphism we have $M(R(\omega)) = \omega$. 
Hence, $\omega$ is rational. Thus the Eisenstein class $\omega$ is rational if and only if  the constant term in the 
Fourier expansion at every cusp is rational. This proves the lemma. 
\end{proof}

\begin{proposition}
  \label{prop rational}
  If $m \neq n$, then the Eisenstein differential form $$E^\beta_{2n-2m+2}(0) \in H^1(\Gamma_0^\beta(N) \backslash \H, \call (2n-2m, E))$$ is rational,
  for a sufficiently large $p$-adic number field $E$. 
\end{proposition}

\begin{proof}

Recall that $\beta = \frac{a \sqrt{-D}}{2p^j}$ if $j \geq 1$ (and $\beta = 0$ if $j = 0$).  We claim that $\Gamma_{0}^{\beta}(N) = \gamma_\beta^{-1} \Gamma_{0}(\mathfrak{N}) \gamma_\beta \cap \SL_{2}(\mathbb{Z})$ is 
independent of $a$. We do this by showing that 
\begin{equation}
  \label{inclusion Gamma beta}
  \Gamma_{0}^{\beta}(N) = \left\{  \left( \begin{smallmatrix} a & b \\ c & d \end{smallmatrix} \right) \in \SL_2(\Z) \> : \> a \equiv d \mod p^j, \>
     c \equiv 0 \mod Np^{2j} \right\}.
\end{equation}
Indeed, if $j = 0$,  \eqref{inclusion Gamma beta} holds trivially, since in this case $\gamma_\beta = 1$, so both sides of   
\eqref{inclusion Gamma beta}  are equal to $\Gamma_0(N)$.  So assume that $j \geq 1$. Since $p$ is odd and we are considering representatives $a \in R = (\mathbb{Z}/p^j\mathbb{Z})^\times/\{\pm 1\}$, 
by replacing $a$ by $p^j - a$ if necessary, we may assume that all $a \in R$ are even, so that $p^j \sqrt{-D}^{-1} \beta \in \mathbb{Z}$.  
Let $\gamma = \begin{psmallmatrix} a & b \\ c & d \end{psmallmatrix} \in \Gamma_{0}(\mathfrak{N})$ 
and $\gamma_{\beta} = \begin{psmallmatrix} 1 & \beta \\ 0 & 1 \end{psmallmatrix}$. 
Then 
\begin{equation}
\label{conjugation by gamma beta}
 \gamma_{\beta}^{-1} \gamma \gamma_{\beta} = 
   \begin{pmatrix}
      a- c \beta &  b-d \beta + (a- c \beta ) \beta  \\ c & d+c\beta
   \end{pmatrix}.
\end{equation}
Assume that the matrix in \eqref{conjugation by gamma beta} is in $\SL_2(\Z)$, so is in $\Gamma_0^\beta(N)$. 
Then $a - c \beta$, $b - d \beta + (a - c \beta)\beta$, $c$, $d + c \beta \in \mathbb{Z}$.
%
Note that  $c \in \mathbb{Z} \Leftrightarrow c \in N\mathbb{Z}$. 
Since $a-c \beta$ and $d +c \beta \in \mathbb{Z}$, we have 
\begin{align*}
 b - d \beta + (a - c \beta)\beta  \in \mathbb{Z} 
& \Leftrightarrow (b + c \beta^2) + (a -c \beta - d - c \beta)\beta \in \Z \\
& \Leftrightarrow 
\Re(b)+c\beta^{2} \in \mathbb{Z} \text{ and } \Im(b)= i (a - c \beta - d - c \beta) \beta \\
&  \Leftrightarrow   p^{2j} \vert c \text{ and }a-c \beta \equiv d+c \beta \mod p^{j},
\end{align*} 
since $p \nmid 2D$ and both $\Re(b)$, $\sqrt{D}^{-1} \Im(b) \in \frac{1}{2}\Z$.  
Therefore $\Gamma_{0}^{\beta}(N)$ is contained in the right hand side of \eqref{inclusion Gamma beta}. 
On the other hand, if $\gamma$ is any matrix on the right hand side of \eqref{inclusion Gamma beta}, then
by replacing $\beta$ by $-\beta$ in \eqref{conjugation by gamma beta},
one checks that  $\gamma_{\beta} \gamma \gamma_{\beta}^{-1} \in \Gamma_{0}(\mathfrak{N})$.
It follows that equality holds in \eqref{inclusion Gamma beta}.

From \eqref{inclusion Gamma beta},  we have 
 $$
      \mathrm{SL}_{2}(\mathbb{Z})_{\infty} = \left\{ \begin{psmallmatrix}
      \pm 1 & n \\ 0 &  \pm1
      \end{psmallmatrix}: n \in \mathbb{Z} \right\} \subset  \Gamma_{0}^{\beta}(N) .
 $$
 Thus $\Gamma_{0}^{\beta}(N)_{\infty} = \mathrm{SL}_{2}(\mathbb{Z})_{\infty}$.
 Also, note that the coset  $\Gamma_{0}^{\beta}(N)_{\infty} \begin{psmallmatrix} a & b \\ c & d  \end{psmallmatrix}$ 
  in   $\Gamma_{0}^{\beta}(N)_{\infty} \backslash \Gamma_{0}^{\beta}(N)$
  contains all the matrices of $\Gamma_{0}^{\beta}(N)$
  whose bottom row equals $\pm (c, d)$. Hence
  $\Gamma_{0}^{\beta}(N)_{\infty} \backslash \Gamma_{0}^{\beta}(N)$
  is in bijection with the set 
  \[
       \Lambda  := \{ (c,d) \in \mathbb{Z}^{2} \smallsetminus \{ (0,0)\}: (c,d) =1,
             c \equiv 0 \mod Np^{2j}, d \equiv \pm 1 \mod p^{j} \} / \{ \pm 1 \}.
  \]
 
 For each integer $k \geq 3$ and $(u,v) \in (\Z/Np^{2j})^2$, consider the Eisenstein series
 $$E_{k}^{(u,v)}(z) := \sum\limits_{\substack{ (c,d) \equiv (u,v) \mod Np^{2j} \\ (c,d) = 1}}^{}   \frac{1}{(cz+d)^{k}}.$$ 
 This 
 Eisenstein series 
 differs from the Eisenstein series in \cite[(4.4)]{DS05} by a factor of $\epsilon_{Np^{2j}} = \frac{1}{2}$ or $1$.
 By \eqref{Eis beta 2}, we have
 \begin{align}  \label{E beta and E series}
    E^{\beta}_{2n-2m+2}(0,z) &= \sum\limits_{(c,d) \in \Lambda}^{} 
    \frac{1}{(cz+d)^{2n-2m+2}} \cdot \omega \nonumber \\ 
    & = \dfrac{1}{2} \sum\limits_{\substack{(u,v) \in (\mathbb{Z} /Np^{2j} \mathbb{Z})^{2}
     \\ u \equiv 0 \mod Np^{2j} \\ v \equiv \pm 1 \mod p^{j}}} E_{2n-2m+2}^{(u,v)}(z) \cdot \omega, 
 \end{align}
noting that $2n-2m+2 \geq 4$, since $m \neq n$. 
 By \cite[(4.6)]{DS05},  for $k \geq 3$, we have
 \begin{align}
 \label{E series and G series}
        E_{k}^{(u,v)}(z) = \sum_{l \in (\mathbb{Z}/Np^{2j}\mathbb{Z})^{\times}}
        \zeta_{+}^{l}(k,\mu) G^{l^{-1}(u,v)}_{k}(z),
 \end{align}
 where 
 \begin{align*}
        \zeta_{+}^{l}(k,\mu) := \sum_{\substack{m=1 \\ m \equiv l \text{ mod } Np^{2j}}}^{\infty} 
        \frac{\mu(m)}{m^{k}},
 \end{align*}
 $\mu(\cdot)$ is the M\"obius function, and
 \begin{align*}
         G^{(u,v)}_{k}(z) :=  
         \sideset{}{'}\sum\limits_{ \substack{ (c,d) \in \mathbb{Z}^{2} \\ (c,d) \equiv (u,v) \text{ mod } Np^{2j}}}^{}  
         \frac{1}{(cz+d)^{-k}}.
 \end{align*}
 We will obtain the $q$-expansion of the Eisenstein series $E_{2n-2m+2}^{\beta}(0,z) = \sum_{n=0}^\infty a_n q^n$, 
using \eqref{E beta and E series}, \eqref{E series and G series} and the $q$-expansion of the Eisenstein series above using facts from \cite{DS05}.
 
 Let $k\geq 2$ an integer and $\varphi$ be the Euler totient function. For $v \in (\Z/Np^{2j})^\times$,  we have 
 \begin{equation}
 \label{zeta}
 \begin{aligned}
        \zeta^{v}(k) & := \sideset{}{'}\sum_{\substack{d \equiv v \text{ mod } Np^{2j} }} d^{-k}  \\
         &=  \sum_{\substack{d =1 \\ d \equiv v \text{ mod } Np^{2j}}}^{\infty} d^{-k } 
          +  (-1)^{k} \sum_{\substack{d = 1 \\ d \equiv -v  \text{ mod } Np^{2j}}}^{\infty} d^{-k }  \\
          & = \frac{1}{\varphi(Np^{2j})} \left( \sum_{\psi \text{ mod } Np^{2j}}  \psi(v)^{-1} L(k, \psi) + (-1)^{k} 
                   \sum_{\psi \text{ mod } Np^{2j}}  \psi(-v)^{-1} L(k, \psi) \right) \\
           &=  \frac{1}{\varphi(Np^{2j})} \left( \sum_{\psi \text{ mod } Np^{2j}} (1+ (-1)^{k} \psi(-1)) \psi(v)^{-1} L(k, \psi)  \right) ,
 \end{aligned}
 \end{equation}
 where the penultimate step follows from \cite[Page 122]{DS05}.  If $k$ is even,
 then  $(1+ (-1)^{k} \psi(-1))$ is equal to $2$ (resp. $0$)  if $\psi$ is even (resp. odd). 
 A similar expression  for  $\zeta_{+}^{l}(k,\mu)$ in terms of 
 Dirichlet $L$-functions can also be derived.
 For  $l \in (\mathbb{Z}/Np^{2j} \mathbb{Z} )^{\times}$, by the orthogonality relations, we have
 \begin{align}
        \label{zeta plus original} 
        \zeta_{+}^{l}(k,\mu) &= \sum_{m=1 }^{\infty}
        \frac{1}{\varphi(Np^{2j})} \sum_{\psi \text{ mod } Np^{2j}} \psi(l)^{-1} \psi(m) \ \mu(m)m^{-k} \nonumber \\
        &=  \frac{1}{\varphi(Np^{2j})} \sum_{\psi \text{ mod } Np^{2j}} \psi(l)^{-1} 
        \sum_{m=1}^{\infty} \psi(m) \ \mu(m)m^{-k} \nonumber \\
         &=  \frac{1}{\varphi(Np^{2j})} \sum_{\psi \text{ mod } Np^{2j}} \psi(l)^{-1}  L(k, \psi)^{-1}, 
 \end{align}
 where the last step follows from by multiplying the corresponding $L$-functions. Therefore
 \begin{align}
 \label{zeta plus}
         \zeta_{+}^{l}(k,\mu) +  \zeta_{+}^{-l}(k,\mu) & =
         \frac{1}{\varphi(Np^{2j})} \sum_{\psi \text{ mod } Np^{2j}} (\psi(l)^{-1} + \psi(-l)^{-1})  L(k, \psi)^{-1} 
         \nonumber \\
         & =  \frac{2}{\varphi(Np^{2j})} 
         \sum_{\substack{\psi \text{ even} \\ \psi \text{ mod } Np^{2j}}} \psi(l)^{-1}   L(k, \psi)^{-1}.
 \end{align}

By \eqref{E beta and E series} and \eqref{E series and G series}, we have 
 \begin{align}
 \label{E beta and G series}
         E_{2n-2m+2}^{\beta}(0,z) & = \frac{1}{2}
         \sum\limits_{\substack{(u,v) \in (\mathbb{Z} /N p^{2j} \mathbb{Z})^{2}
          \\ u \equiv 0 \text{ mod } Np^{2j} \\ v \equiv \pm 1 \text{ mod } p^{j}}} ~~
          \sum_{l \in (\mathbb{Z}/Np^{2j}\mathbb{Z})^{\times}}
        \zeta_{+}^{l}(2n-2m+2,\mu) G^{l^{-1}(u,v)}_{2n-2m+2}(z) \cdot \omega. 
 \end{align}
 For simplicity, let $k =2n-2m+2$. By the description of the set $\Lambda$, we have $(c,d) \in \Lambda$ implies that 
 $N p^{2j} \mid c$ and $(c,d)=1$, so the congruence class $v$ of $d$ mod $Np^{2j}$
 has order $Np^{2j}$.  Therefore, for $m < n$, by \cite[Theorem 4.2.3]{DS05}, we have
 \begin{align}
 \label{Fourier expansion G series}
        G^{(u',v')}_{k}(z) =  \zeta^{v'}(k) +  
            \frac{(-2 \pi i)^{k}}{(k-1)! (Np^{2j})^{k}} 
            \sum_{l=1}^{\infty} \sigma_{k-1}^{(u',v')} (l) e^{2 \pi i l z/Np^{2j}},
 \end{align}
 for tuples $(u',v')$ occurring in \eqref{E beta and G series}, where
 \begin{align*}
         \sigma_{k-1}^{(u',v')} (l) = 
         \sum_{\substack{ l' \mid l \\  l/l' \equiv 0 \text{ mod } Np^{2j}}}^{} \mathrm{sgn}(l') l'^{k-1}
         e^{2 \pi i v' l'/ Np^{2j}}.
 \end{align*}

 \vspace{.1cm}
 \noindent \textbf{Constant term: }
 \vspace{.1cm}

 \noindent By \eqref{E beta and G series} and \eqref{Fourier expansion G series}, the constant term
 $a_{0}$ in the $q$-expansion   of $E^{\beta}_{k}(0,z)$ equals
\begin{align*}
      a_{0}  & =  \frac{1}{2} \sum\limits_{\substack{(u,v) \in (\mathbb{Z} /N p^{2j} \mathbb{Z})^{2}
          \\ u \equiv 0 \text{ mod } Np^{2j} \\ v \equiv \pm 1 \text{ mod } p^{j}}} ~~
          \sum_{l \in (\mathbb{Z}/Np^{2j}\mathbb{Z})^{\times}}
           \zeta_{+}^{l}(k,\mu)   \zeta^{l^{-1}v}(k) \\
           & \stackrel{(*)}{=}  \sum\limits_{\substack{(u,v) \in (\mathbb{Z} /N p^{2j} \mathbb{Z})^{2}
          \\ u \equiv 0 \text{ mod } Np^{2j} \\ v \equiv \pm 1 \text{ mod } p^{j}}} ~~
          \sum_{l \in (\mathbb{Z}/Np^{2j}\mathbb{Z})^{\times}}
          \frac{1}{\varphi(Np^{2j})^{2}} 
         \sum_{\psi \text{ mod } Np^{2j}} \psi(l)^{-1}   L(k, \psi)^{-1} 
         \sum_{\substack{\psi_1 \text{ mod } Np^{2j} \\ \psi_{1} \text{ even}}} \psi_{1}(l^{-1}v)^{-1}   L(k, \psi_{1}) \\
         & =     \frac{1}{\varphi(Np^{2j})^{2}}  
          \sum\limits_{\substack{(u,v) \in (\mathbb{Z} /N p^{2j} \mathbb{Z})^{2}
           \\ u \equiv 0 \text{ mod } Np^{2j} \\ v \equiv \pm 1 \text{ mod } p^{j}}} ~~
            \sum_{\substack{\psi, \psi_{1} \text{ mod } Np^{2j} \\  \psi_{1} \text{ even}}} 
           \psi_{1}(v)^{-1}  \frac{L(k,\psi_{1})}{L(k, \psi)}  
          \sum_{l \in (\mathbb{Z}/Np^{2j}\mathbb{Z})^{\times}} \psi_{1} \psi^{-1}(l) \\
          & \stackrel{(**)}{=}
           \frac{1}{\varphi(Np^{2j})} \sum\limits_{\substack{(u,v) \in (\mathbb{Z} /N p^{2j} \mathbb{Z})^{2}
           \\ u \equiv 0 \text{ mod } Np^{2j} \\ v \equiv \pm 1 \text{ mod } p^{j}}} ~~
            \sum_{\substack{\psi_1 \text{ mod } Np^{2j} \\ \psi_1 \text{ even}}} 
           \psi_1(v)^{-1}  \\
           & =   \frac{1}{2\varphi(Np^{2j})}
            \sum\limits_{\substack{(u,v) \in (\mathbb{Z} /N p^{2j} \mathbb{Z})^{2}
           \\ u \equiv 0 \text{ mod } Np^{2j} \\ v \equiv \pm 1 \text{ mod } p^{j}}} ~~
            \sum_{\psi_1 \text{ mod } Np^{2j}}  ( \psi_1(v)^{-1}  + \psi_1(-v)^{-1} ) \\
           & \stackrel{(***)}{=} 1,
\end{align*} 
where $(*)$ follows from \eqref{zeta} and \eqref{zeta plus original}, and $(**)$ and $(***)$ follow from the orthogonality relations.

\vspace{.1cm}
\noindent\textbf{Higher Fourier coefficients:}
\vspace{.1cm}

\noindent Clearly $\sigma_{k-1}^{(0,v')} (l)=0$ if $Np^{2j} \nmid l$. So assume that
$l$ is a multiple of $Np^{2j}$. Say $l = Np^{2j}l''$. Then
\[
\sigma_{k-1}^{(0,v')} (l)= \sum_{l' \mid l''}^{} \mathrm{sgn}(l') l'^{k-1} e^{2 \pi i v' l'/ Np^{2j}},
\]
which is clearly $E$-rational if $E$ contains a sufficiently large cyclotomic number field depending on $j$.  
From \eqref{E beta and G series} and \eqref{Fourier expansion G series}, 
we see that the coefficient  $a_{l''}$ of $q^{l''}$ in the Fourier expansion 
of $E_{k}^{\beta}(0,z)$ equals
\begin{align*}
a_{l''}  &= \frac{1}{2} \sum\limits_{\substack{(u,v) \in (\mathbb{Z} /N p^{2j} \mathbb{Z})^{2}
         \\ u \equiv 0 \text{ mod } Np^{2j} \\ v \equiv \pm 1 \text{ mod } p^{j}}} ~~
         \sum_{n \in (\mathbb{Z}/Np^{2j}\mathbb{Z})^{\times}}
          \zeta_{+}^{n}(k,\mu)  \frac{(-2 \pi i)^{k}}{(k-1)! (Np^{2j})^{k}}
          \sigma_{k-1}^{n^{-1}(0,v)} (Np^{2j}l'').
\end{align*}
If $j = 0$, one checks that the formula for $a_{l''}$ above reduces to a well-known expression (see \cite[Theorem 7.1.3 and (7.1.30)]{Miy89}), and
in particular $a_{l''} \in \mathbb{Q}$ is rational. So assume that $j > 0$. 
Then 
\begin{align*}
a_{l''}       &= \frac{1}{2} \sum\limits_{\substack{(u,v) \in (\mathbb{Z} /N p^{2j} \mathbb{Z})^{2}
         \\ u \equiv 0 \text{ mod } Np^{2j} \\ v \equiv 1 \text{ mod } p^{j}}} ~~
         \sum_{n \in (\mathbb{Z}/Np^{2j}\mathbb{Z})^{\times}}
          ( \zeta_{+}^{n}(k,\mu) + \zeta_{+}^{-n}(k,\mu) ) \frac{(-2 \pi i)^{k}}{(k-1)! (Np^{2j})^{k}}
          \sigma_{k-1}^{n^{-1}(0,v)} (Np^{2j}l'')  \\
            & \stackrel{\eqref{zeta plus}}{=} \frac{1}{\varphi(Np^{2j})}   \sum\limits_{\substack{(u,v) \in (\mathbb{Z} /N p^{2j} \mathbb{Z})^{2}
         \\ u \equiv 0 \text{ mod } Np^{2j} \\ v \equiv 1 \text{ mod } p^{j}}} ~~
            \sum_{n \in (\mathbb{Z}/Np^{2j}\mathbb{Z})^{\times}}
             \sum_{\substack{\psi \text{ mod } Np^{2j} \\ \psi \text{ even} }}  \psi(n)^{-1}   L(k, \psi)^{-1}
             \frac{(-2 \pi i)^{k}}{(k-1)! (Np^{2j})^{k}} \\
            & \qquad \qquad \qquad \qquad \qquad \qquad \qquad \qquad \qquad \qquad \qquad \qquad \qquad \qquad  \cdot   \sigma_{k-1}^{n^{-1}(0,v)} (Np^{2j}l'') \\   
              & \stackrel{\eqref{functional equation}}{=}  \frac{-2k}{\varphi(Np^{2j})}  \sum\limits_{\substack{(u,v) \in (\mathbb{Z} /N p^{2j} \mathbb{Z})^{2}
         \\ u \equiv 0 \text{ mod } Np^{2j} \\ v \equiv 1 \text{ mod } p^{j}}} ~~
              \sum_{n \in (\mathbb{Z}/Np^{2j}\mathbb{Z})^{\times}}
             \sum_{\substack{\psi \text{ mod } Np^{2j} \\ \psi \text{ even} }} \psi(n)^{-1} 
             \left(\frac{C_{\psi}}{Np^{2j}} \right)^{k}
             \frac{1} {G(\psi) B_{k,\bar{\psi}}} \\
             & \qquad \qquad \qquad \qquad \qquad \qquad \qquad \qquad \qquad \qquad \qquad \qquad \qquad \qquad  \cdot  \sigma_{k-1}^{n^{-1}(0,v)} (Np^{2j}l''),   
\end{align*}
where in the last step we have used 
the following special value result for the Dirichlet $L$-function:
\begin{align}
\label{functional equation}
 L(k, \psi) &=  - \frac{(-2 \pi i)^{k} G(\psi) B_{k, \bar{\psi}}}{2k! C_{\psi}^{k}}
  \text{ if } \psi \text{ is even and } k >0 \text{ is even},
\end{align}
where $C_{\psi}$  denotes the conductor of $\psi$. Thus $a_{l''}$ is again $E$-rational for a sufficiently large $p$-adic number field $E$ containing
an appropriate cyclotomic number field.  

Summarizing, the computations above show that $E^\beta_{2n-2m+2}(0,z)$ has an $E$-rational $q$-expansion $\sum_{n=0}^\infty a_n q^n$ (at the cusp $\infty$) 
if $E$ contains a sufficiently large cyclotomic number field (which depends on $j$).  
By \cite[Proposition 4.2.1]{DS05}, since
 $E_{2n-2m+2}^{(u,v)} \vert_{\gamma} =  E_{2n-2m+2}^{(u,v)\gamma}$, for all $\gamma \in \SL_2(\Z)$, the Eisenstein series
 $E^\beta_{2n-2m+2}(0,z)$ has an $E$-rational $q$-expansion at 
 each cusp $\xi$ of $\Gamma_0^\beta(N)$. The proposition now follows from Lemma~\ref{lemma rational}.
\end{proof}

\section{Towards integrality}

Note that the map $T_\beta^*|_{\H}$ can also be described as the pull-back of a differential form via the map
$$
S_\beta: \Gamma_0^\beta (N) \backslash \H \to \Gamma_0 (\fr n) \backslash \calh
$$
given by sending
$$
x + it \mapsto \gamma_\beta \bmatrix x & -t \\ t & x \endbmatrix.
$$

We now choose $\delta^\circ (f)$ such that it generates 
$\bar{H}^1_{cusp} (\Gamma_0 (\fr n) \backslash \calh, \call (\uln, \calo_E))[f]$, which is a rank one $\calo_E$-submodule of 
$\bar{H}^1_{cusp} (\Gamma_0 (\fr n) \backslash \calh, \call (\uln, \calo_E))$, where 
$\calo_E$ is the valuation ring of $E$ and $\bar{H}^1$ denotes the image of the integral cohomology in the rational 
cohomology under the natural map. We
correspondingly refine the period $\Omega (f)$ so that $\Omega(f) \in \C^\times/\calo_E^\times$. 

Since $\beta = \frac{a \sqrt{-D}}{2 p^j}$, we have $\gamma_\beta^{-1} \cdot P \in L(\uln,  \frac{1}{p^{2nj}} \calo_E)$, for $P \in L(\uln, \calo_E)$.
Thus, the map $S_\beta$ does not preserve cohomology with integral coefficients, but instead induces a map
\begin{equation}
  \label{S beta integrally}
S_\beta^*: \bar{H}^1_{cusp} (\Gamma_0 (\fr n)\backslash \calh, \call (\uln, \calo_E)) \to  \bar{H}^1_{cusp} (\Gamma^\beta_0 (N) \backslash \H, \call (\uln, \frac{1}{p^{2nj}} \calo_E)),
\end{equation}
on cohomology. 

\begin{lemma} 
  \label{lemma Clebsch Gordan} 
  Assume $p > n$. Then under the Clebsch-Gordan decomposition \eqref{Clebsch Gordan}, we have  
  $$S_\beta^*( \bar{H}^1_{cusp} (\Gamma_0 (\fr n) \backslash \calh, \call (\uln, \calo_E)) \longmapsto \bigoplus_{m=0}^n 
  \bar{H}^1_{cusp} (\Gamma_0^\beta(N)\backslash \H, \call (2n-2m, \frac{1}{p^{j(2n-m)}} \calo_E)).$$
\end{lemma}

\begin{proof} 
Let $\nabla =  \biggl( {\partial^2 \over \partial X \partial \overline Y} -
                     {\partial^2 \over \partial \overline X \partial Y} \biggr)$.
By \cite[Lemma 2]{Gh}, the projection to the $m$-th component in \eqref{Clebsch Gordan} is induced by 
$P (X,Y, \overline X, \overline Y)    \mapsto   \frac{1}{{m!}^2}
    \nabla^m
     P(X, Y, \overline X, \overline Y) \> \vert_{\overline X  = X \atop \overline Y = Y}.$ 
Clearly the projection continues to be defined  with $\calo_E$ coefficients if $p > n$. As remarked in \eqref{S beta integrally}, $S_\beta^*$ does not 
preserve integrality. However, since 
\begin{multline*}
  \nabla (\gamma_\beta^{-1} \cdot X^{n-k}Y^k {\overline X}^{n-l} {\overline Y}^l) = \\ 
  l {\partial \over \partial X} X^{n-k} (Y - \beta X)^{k} {\overline X}^{n-l} ({\overline Y  + \beta \overline X})^{l-1} 
- k {\partial \over \partial \overline X} X^{n-k} (Y - \beta X)^{k-1} {\overline X}^{n-l} ({\overline Y  + \beta \overline X})^l,
\end{multline*}
we see that if  $P \in L(\uln, \calo_E)$, the total power of $p^j$ in the denominator goes down by one after applying 
$\nabla$ to $\gamma^{-1}_\beta \cdot P$. Iterating this, we see $\nabla^m (\gamma^{-1}_\beta \cdot P) \in  L (2n-2m, \frac{1}{p^{j(2n-m)}} \calo_E)$, for 
$m = 0, \dots, n$, proving the lemma. 
\end{proof}

We now assume that the prime $p$ is greater than $n$, so that we may apply the lemma above. Let 
\begin{eqnarray}
  \label{S beta denom}
   S_\beta^*(\delta^\circ_{2n-2m}(f)) \in  \bar{H}^1_{cusp} (\Gamma_0^\beta(N)\backslash \H, \call (2n-2m, \frac{1}{p^{j(2n-m)}} \calo_E))
\end{eqnarray} 
be the image of $\delta^\circ(f)$ under the map \eqref{S beta integrally} followed by the projection to the $m$-th component in the 
Clebsch-Gordan decomposition in Lemma~\ref{lemma Clebsch Gordan}. Again note the slight abuse of notation, since the subscript $2n-2m$ should
be outside the brackets.

 
By Proposition~\ref{prop rational}, we know that 
\begin{eqnarray}
  \label{E beta denom} 
  E^\beta_{2n-2m+2}(0) \in  \frac{1}{p^{c_j}} \bar{H}^1(\Gamma_0^\beta(N) \backslash \H, \call (2n-2m, \calo_E)),
\end{eqnarray}
for some integer $c_j \geq 0$, depending on $j$.

Let $S$ denote the finite set of excluded primes above (i.e., $p \mid 2ND$ and $p \leq n$), which we extend to include the primes $p < 2n+4$. 
 We remark that if $p \not\in S$, then $p >2n$ which ensures that the duality pairing
$\langle \>,\> \rangle$ is a well-defined pairing on cohomology with integral coefficients $L(2n-2m,\calo_E)$. 

For the refined period $\Omega(f)$ defined above, 
we get the following partial integrality result.

\begin{proposition} Suppose $p$ is not in the finite set of primes $S$, and that $E$ is a sufficiently large $p$-adic number field as above.
Let $0 \leq m < n$ be even and $\chi$ be an even character of conductor $p^{j_\chi}$.  Then 
  \label{theorem integrality} 
  \begin{equation}\label{eq: integrality}
    \frac{G(\chi) G(2n-m+2, \bar \chi, f)}{\Omega(f) G(\bar \chi^2) \Omega_\infty}  \in \frac{ \calo_E}{p^{j_\chi(4n-3m+3)+c_{j_\chi}}}.
\end{equation}
\end{proposition}

\begin{proof}
Indeed, this follows from the fact that by \eqref{integral-expression} the special value in the statement of the proposition has a 
cohomological description in terms of integrals of the form
$$
\int_{[\Gamma^\beta_0 (N) \backslash \H]} S^*_\beta \delta_{2n-2m}^\circ (f) \wedge E_{2n-m+2}^\beta.
$$
These are integrals of cohomology classes with specifiable denominator over an integral cycle, hence belong to $\calo_E$ with specifiable denominator. 
The size of the denominator can be computed from \eqref{S beta denom} and \eqref{E beta denom}, taking $j = j_\chi$, and the fact
that the Dirichlet $L$-value in \eqref{integral-expression} satisfies $L^\circ(2n-2m+2, \bar\chi^2) \in \frac{1}{p^{j_\chi(2n-2m+3)}} \calo_E$
(which in turn follows easily from a special values result like \eqref{functional equation}, noting 
that the corresponding twisted Bernoulli numbers lie in $\frac{1}{p^{j_\chi}} \calo_E$ by a standard formula for these numbers involving the usual Bernoulli
numbers up to $B_{2n+2}$, and by the well-known result of von Staudt-Clausen which says that $p$ does not divide the denominators of these 
Bernoulli numbers since $p-1 > 2n+2$, by the definition of the set $S$).
\end{proof}

\section{Constructing bounded distributions}
  \label{section bounded distributions}

Finally, we now define our $p$-adic distribution by the formula
$$
\mu^\circ_{2n-m+2} = \frac{1}{\Omega(f) \Omega_\infty} \cdot \mu_{2n-m+2}.
$$
These distributions are certainly defined whenever $2n-m+2 \ge k+2$ which is the same as $m \le n-2$, but may possibly
be defined for all $0 \leq m \leq n$, by analytic continuation. 

We wish to show that $\mu^\circ_{2n-m+2}$ is a bounded distribution and hence a measure. To this end we recall 
the notion of abstract Kummer congruences. 

\begin{theorem}[Abstract Kummer congruences]
\label{thm: AKC} 
 Let $Y = \Z_p^\times$, let $\calo_p$ be the ring of integers of $\C_p$ and let $\{f_i\}$
be a collection of continuous functions in
$C(Y, \calo_p)$ such that the $\C_p$-linear span of $\{f_i\}$ is dense in $C(Y, \C_p)$. Let $\{a_i\}$ be a system of elements 
with $a_i \in \calo_p$. Then the existence of an $\calo_p$-valued measure $\mu$ on $Y$ with the property
$$\int_Y \>  f_i \> d\mu = a_i$$
is equivalent to the following congruences: for an arbitrary choice of elements $b_i \in \C_p$ almost all zero, and
for $n \geq 0$, we have
     $$ \sum_i b_i f_i (y) \in p^n \calo_p, \text{ for all }y \in Y \implies \sum_i b_i a_i \in p^n \calo_p. $$ 
\end{theorem}

We apply this theorem with $f_i$ the collection of Dirichlet characters $\chi$ of $(\Z/p^j\Z)^\times$, 
for all $j \geq 1$, thought of as functions of $Y = \Z_p^\times$, and with 
$a_\chi \in \calo_p$ the values of $\mu(\chi)$, for a given $\C_p$-valued distribution $\mu$ on $Y$. 
To prove that $\mu$ is an $\calo_p$-valued measure on $Y$, it suffices to prove Kummer congruences of the more specialized form 
\begin{equation}\label{eq: Kummer}
\sum_\chi \chi^{-1}(a) \chi (y)   \in p^{j-1} \calo_p, \text{ for all } y \in Y \implies 
      \sum_\chi \chi^{-1}(a) \mu(\chi)  \in p^{j-1} \calo_p,
\end{equation}
where $\chi$ varies over all characters mod $p^j$, for a fixed $j \geq 1$, and where the first congruence 
in \eqref{eq: Kummer} follows from the identity $ \sum_\chi \chi^{-1}(a) \chi  = \phi(p^j) \id_{a + p^j \Z_p} $, for
$\id_{a + p^j \Z_p}$ the characteristic function of the coset $a + p^j \Z_p \subset \Z_p^\times$. Indeed, then the second
congruence in \eqref{eq: Kummer} shows that  
$\mu$ is $\calo_p$-valued on $\id_{a + p^j \Z_p}$, whence on all $\calo_p$-valued step functions on $\Z_p$, whence on 
all $\calo_p$-valued continuous functions on $Y$. 

\begin{claim} 
  The Kummer congruences \eqref{eq: Kummer} hold for $\mu = \mu^\circ_{2n-m+2}$, for $m \leq n - 2$ even. 
\end{claim}

In order to prove this claim we must show that the second sum 
$\sum_\chi \chi^{-1}(a) \mu^\circ_{2n-m+2}(\chi)$ in \eqref{eq: Kummer} 
should firstly a) be {\it integral} and secondly b) be in $p^{j-1} \calo_p$. 
Now \eqref{eq: integrality} shows that for any even 
character $\chi$ and $m$ even:
\begin{equation}\label{eq: interpolation}
\int \chi\ d\mu^\circ_{2n-m+2} = \frac{2 p^{j_\chi (2n-m+1) }G(\bar \chi^2) }{\kappa^{j_\chi}}  \frac{G(\chi) G(2n-m+2, \bar \chi, f)}{\Omega(f) G(\bar \chi^2) \Omega_\infty} \in \frac{1}{p^{j_\chi(2n-2m+2)+c_{j_\chi}}} \calo_E,
\end{equation}
at least if $\kappa$ is a unit, which we have assumed. For odd characters $\chi$, the integral above vanishes. Thus \eqref{eq: interpolation} shows
that the second sum above is in $\frac{1}{p^{j(2n-2m+2)+c_j}} \calo_p$, with $c_j = \max_\chi c_{j_\chi}$.  This is still quite 
far from the integrality claimed in part a). Assuming that part a) holds, one must then further prove the congruence in part b).

In any case, assuming the Claim, we have that  $\mu^\circ_{2n-m+2}$ is a measure, for $0 \leq m \leq n-2$ even. 


Let $x_p :  \Z_p^\times \rightarrow \calo_p$ be the usual embedding. 
We now wish to glue the measures $\mu^\circ_{2n-m+2}$, for $0 \leq m$ even, into one measure $\mu^\circ$ satisfying 
(see \cite[Lemma 4.4]{Coates-PR}, noting $q(V)$ there equals $1$)
\begin{equation}
  \label{eq: glue}
  \int_{\Z_p^\times} \chi  \ d\mu^\circ = (-1)^{m/2}  \int_{\Z_p^\times} x_p^m \chi \ d\mu^\circ_{2n - m +2}.
\end{equation}
To do this, we again appeal to the abstract Kummer congruences in the theorem above.
For the  $f_i$, we consider a slightly larger class of functions than the Dirichlet characters $\chi$ above, namely those 
of the form $x_p^{-m} \cdot \chi$, for $0 \leq m \leq n$, with $m$ even. We set $a_{m,\chi} = (-1)^{m/2} \mu^\circ_{2n-m+2}(\chi) \in \calo_p$,
which should be equal to $\mu^\circ(x_p^{-m} \chi)$, by \eqref{eq: glue} above.
We now assume that 

\begin{claim}
  The $a_{m,\chi}$ satisfy the abstract Kummer congruences: 
  $$ \sum_{m, \chi} \ b_{m, \chi} (x_p^{-m} \chi) (y) \in p^{j-1} \calo_p, \text{ for all } y \in Y  \implies \sum_{m,\chi} b_{m,\chi} a_{m, \chi} \in p^{j-1} \calo_p.$$
\end{claim}

It would then follow from Theorem~\ref{thm: AKC} that there is a measure $\mu^\circ$ such that \eqref{eq: glue} holds. Note that the Kummer congruences 
in the latter claim actually imply the ones in the former claim for $\mu_{2n-m+2}^\circ$, 
by choosing the $b_{m', \chi} = \chi^{-1}(a)$ if $m'=m$, and $b_{m', \chi} = 0$ 
if $m' \neq m$.   We expect that the proof of these Kummer congruences should be similar to the Kummer congruences 
proved by Panchishkin in his construction of the $p$-adic Rankin product $L$-function attached to two cusp forms $f$ and $g$, described
in detail in \cite{Pan} (see also \cite{CouPan}, and \cite{Ghate-Vangala} where a sign similar to the one occurring in \eqref{eq: glue} is corrected).  

Since $\mu^\circ$ and $\mu^\circ_{2n+2}$ agree on a dense set of functions, namely all $\chi$, the measure 
$\mu^\circ$ is just the measure $\mu^\circ_{2n+2}$.      
We now define the $p$-adic Asai $L$-function as the Mellin transform of the measure  $\mu^\circ = \mu^\circ_{2n+2}$:
$$
L_p (\chi) = \int_{\Z_p^\times} \chi(a) \ d\mu^\circ_{2n+2}, \quad \text{ for all }  \chi: \Z_p^\times \to \C_p^\times.
$$

\hskip 1cm

{\bf \noindent Acknowledgements:} The first author was supported by SERB grants EMR/2016/000840 and MTR/2017/000114. The second and third authors thank 
T.N. Venkataramana for useful conversations. A version of this paper has existed since about 2016. Recently, Loeffler and Williams \cite{LW19}
have announced a construction of the $p$-adic Asai $L$-function attached to a Bianchi cusp form using a method involving Euler systems.


\begin{thebibliography}{999999}

\bibitem[CP89]{Coates-PR}  John Coates and  Bernadette Perrin-Riou. {\it On $p$-adic $L$-functions attached to motives over $\Q$}.  Algebraic number theory, 
 Adv.\,Stud.\,Pure Math. \textbf{17} (1989), Academic Press, 23--54,
 
 
\bibitem[CP04]{CouPan}  Michel Courtieu and Alexei  Panchishkin. {\it Non-Archimedean $L$-functions and arithmetical Siegel modular
 forms}. Second edition.
Lecture Notes in Mathematics \textbf{1471}, Springer-Verlag, Berlin (2004).

\bibitem[DS05]{DS05} Fred Diamond and Jerry Shurman.
      \textit{A first course in modular forms}. Graduate Texts in Mathematics \textbf{228}, Springer-Verlag, New York (2005).

      
\bibitem[Gha99]{Gh} Eknath Ghate. {\it Critical values of the twisted tensor $L$-function in the imaginary quadratic case}. Duke Math. J. \textbf{96}  (1999),  no. 3, 595--638.

\bibitem[Gh99b]{Ghb} Eknath Ghate. {\it Critical values of twisted tensor L-functions over CM-fields}. Automorphic forms, automorphic representations, and 
arithmetic (Fort Worth, TX, 1996), Proc. Sympos. Pure Math. \textbf{66}, Part 1, Amer. Math. Soc., Providence, RI (1999),  87--109.

\bibitem[GV19]{Ghate-Vangala} Eknath Ghate and Ravitheja Vangala. {\it Non-Archimedean Rankin $L$-functions}. Ramanujan Math. Society, Lecture Note Series (2019), 
29 pp., to appear.  
 
\bibitem[Har81]{Ha4} G\"unter Harder. {\it Period integrals of {E}isenstein cohomology classes and special values of some ${L}$-functions}. Progr. Math. \textbf{26} (1982), 103--142.

\bibitem[Har87]{Ha} G\"unter Harder. {\it Eisenstein cohomology of arithmetic groups. The case $\mathrm{GL}_2$}. Invent. Math. \textbf{89} (1987), no. 1, 37--118.

\bibitem[Hid94]{Hi} Haruzo Hida. {\it On the critical values of ${L}$-functions of $\mathrm{GL}_2$ and $\mathrm{GL}_2 \times \mathrm{GL}_2$}. Duke Math. J. \textbf{74} (1994), no. 2, 431--529.

\bibitem[LSO14]{sl} Dominic Lanphier and Howard Skogman. {\it Values of twisted tensor $L$-functions of automorphic forms over imaginary quadratic fields}.
With an appendix by Hiroyuki Ochiai. Canad. J. Math. \textbf{66} (2014), no. 5, 1078--1109.

\bibitem[LW19]{LW19} David Loeffler and Chris Williams. {\it $P$-adic Asai $L$-functions of Bianchi modular forms}. https://arxiv.org/abs/1802.08207 (2019), 34 pp.

\bibitem[Miy89]{Miy89} Toshitsune Miyake. {\it Modular forms}. Springer-Verlag, Berlin (1989).

\bibitem[Pan88]{Pan} Alexei Panchishkin. {\it Non-Archimedean Rankin $L$-functions and their functional equations}. Izv. Akad. Nauk SSSR Ser. Math. \textbf{52} (1988), no. 2, 336--354.


\end{thebibliography}
\end{document}